\documentclass[reqno, 11pt, letterpaper]{amsart}
\usepackage{amsfonts,amsmath,amssymb,color}

\numberwithin{equation}{section}

\usepackage[kerning=true]{microtype}

\usepackage{amsthm,leqno}
\usepackage{hyperref}
\usepackage{bigints}
\usepackage{dsfont}

\usepackage{amsfonts,amsmath,amssymb,enumerate}
\usepackage{graphicx}
\usepackage{tikz}																
\usetikzlibrary{decorations.markings}
\usetikzlibrary{plotmarks}
\usetikzlibrary{patterns}

\newtheorem{theo}{Theorem}[section]
\newtheorem{prop}[theo]{Proposition}

\newtheorem{lemme}[theo]{Lemma}

\theoremstyle{remark}
\newtheorem{remark}[theo]{Remark}

\renewcommand{\(}{\left(}
\renewcommand{\)}{\right)}					

\newcommand{\be}{\begin{equation*}}
\newcommand{\ee}{\end{equation*}}
\newcommand{\ben}{\begin{equation}}
\newcommand{\een}{\end{equation}}
\newcommand{\begincal}{\begin{eqnarray*}}
\newcommand{\fincal}{\end{eqnarray*}}
\newcommand{\bal}{\begin{aligned}}
\newcommand{\eal}{\end{aligned}}

\newcommand{\pui}{\frac{n-2}{2}}

\newcommand{\ua}{u_\alpha}

\newcommand{\va}{v_\alpha}

\newcommand{\xia}{x_{i,\alpha}}
\newcommand{\xja}{x_{j,\alpha}}

\newcommand{\Via}{V_{i,\alpha}}

\newcommand{\Wia}{W_{i,\alpha}}
\newcommand{\Bia}{B_{i,\alpha}}
\newcommand{\Bja}{B_{j,\alpha}}
\newcommand{\mWa}{\mathcal{W}_{\alpha}}

\newcommand{\xa}{x_\alpha}

\newcommand{\sa}{\sigma_\alpha}

\newcommand{\ma}{\mu_\alpha}

\newcommand{\nua}{\nu_\alpha}
\newcommand{\mia}{\mu_{i,\alpha}}
\newcommand{\mja}{\mu_{j,\alpha}}

\newcommand{\ga}{\gamma_\alpha}

\newcommand{\lija}{\lambda_{i,j}^{\alpha}}
\newcommand{\ria}{r_{i,\alpha}}
\newcommand{\rja}{r_{j,\alpha}}
\newcommand{\sjia}{s_{j,i,\alpha}}
\newcommand{\sija}{s_{i,j,\alpha}}
\newcommand{\rjia}{\rho_{j,i,\alpha}}

\newcommand{\tia}{\theta_{i,\alpha}}

\newcommand{\Tia}{\Theta_{i,\alpha}}
\newcommand{\Tja}{\Theta_{j,\alpha}}
\newcommand{\tja}{\theta_{j,\alpha}}
\newcommand{\Oia}{\Omega_{i,\alpha}}
\newcommand{\Xia}{\Xi_{i,\alpha}}
\newcommand{\uoa}{u_{0, \alpha}}

\newcommand{\vea}{\varepsilon_\alpha}
\newcommand{\ea}{\eta_\alpha}
\newcommand{\ta}{\tau_\alpha}

\newcommand{\psia}{\psi_\alpha}
\newcommand{\vpa}{\varphi_\alpha}
\newcommand{\hvpa}{\hat{\varphi}_\alpha}
\newcommand{\psa}{\psi_\alpha}

\newcommand{\R}{\mathbb{R}}
\newcommand{\ve}{\varepsilon}
\newcommand{\vp}{\varphi}

\numberwithin{equation}{section}

\title[]{A priori estimates for finite-energy sign-changing blowing-up solutions of critical elliptic equations}

\author{Bruno Premoselli}

\address{Bruno Premoselli, Universit\'e Libre de Bruxelles, Service d'Analyse CP 214, Boulevard du Triomphe, B-1050 Bruxelles, Belgique.}
\email{bruno.premoselli@ulb.be}

\begin{document}

\begin{abstract}
We prove sharp pointwise blow-up estimates for finite-energy sign-changing solutions of critical equations of Schr\"odinger-Yamabe type on a closed Riemannian manifold $(M,g)$ of dimension $n \ge 3$. This is a generalisation of the so-called $C^0$-theory for positive solutions of Schr\"odinger-Yamabe type equations. To deal with the sign-changing case we develop a method of proof that combines an \emph{a priori} bubble-tree analysis with a finite-dimensional reduction, and reduces the proof to obtaining sharp \emph{a priori} blow-up estimates for a linear problem. 
\end{abstract}

\maketitle

\section{Introduction and statement of the results}

\subsection{Pointwise estimates for finite-energy blowing-up sequences of solutions}

Let $(M,g)$ be a smooth and closed Riemannian manifold of dimension $n \ge 3$, where closed means compact without boundary. Let $(h_\alpha)_\alpha$ be a sequence of bounded functions in $M$. We are interested in this work in the following Schr\"odinger-Yamabe type equation in $M$:
\[ \triangle_g u + h_\alpha u = |u|^{2^*-2}u, \]
 where $\triangle_g = - \textrm{div}_g(\nabla \cdot)$ is the Laplace-Beltrami operator and $2^* = \frac{2n}{n-2}$ is the critical exponent for the embedding of the Sobolev space $H^1(M)$ into Lebesgue spaces. This equation originated a vast amount of work and has been the target of active investigation for decades. 
 
\medskip
 
Let $(\ua)_\alpha$, with $\ua \in H^1(M)$, be a sequence of solutions of
\ben \label{eqha}
\triangle_g \ua + h_\alpha \ua = |\ua|^{2^*-2} \ua 
\een
in $M$. Unless specified otherwise, the $\ua$'s are allowed to change sign. Since $h_\alpha$ is bounded the regularity result of Trudinger \cite{Trudinger} ensures that $\ua \in C^{1,\beta}(M)$ for all $0 < \beta < 1$. Following standard terminology we say that $(\ua)_\alpha$ \emph{blows-up with finite energy} if 
\ben \label{blowup}
 \Vert \ua \Vert_{H^1(M)} \le C  \quad \textrm{ and } \quad \lim_{\alpha \to + \infty} \Vert \ua \Vert_{L^\infty(M)} = + \infty \een
for some $C >0$ independent of $\alpha$. The asymptotic behavior of $(\ua)_\alpha$ in the energy space $H^1(M)$ has been known since Struwe's seminal work \cite{Struwe}: assume e.g. that $h_\alpha$ converges in $L^2(M)$ towards $h$. Then, up to a subsequence, $(\ua)_\alpha $ decomposes as 
\ben \label{struwe}
 \ua = u_0 + \sum_{i=1}^k \Via + \phi_\alpha,
 \een
where $\Vert \phi_\alpha \Vert_{H^1(M)} \to 0$ as $\alpha \to + \infty$, $u_0$ is a weak solution of the limiting equation
\[\triangle_g u_0 + h u_0 = |u_0|^{2^*-2} u_0 \]
in $M$, $k\ge 1$ is an integer and, for all $1 \le i \le k$, $\Via$ is a \emph{sign-changing bubble}. More precisely, there exists a nontrivial, possibly sign-changing finite-energy solution $V_i$ of the euclidean Yamabe equation in $\R^n$:
\be
\triangle_{\xi} V_i = |V_i|^{2^*-2} V_i
\ee
and there exist a sequence $(\xia)_\alpha$ of points of $M$ and a sequence $(\mia)_\alpha$ of positive numbers with $\mia \to 0$ as $\alpha \to + \infty $ such that 
\[ \Via(x) \approx \mia^{1 - \frac{n}{2}} V_i \Big( \frac{1}{\mia} \exp_{\xia}^{-1}(x) \Big) \]
at first order, where $\exp_{\xia}$ is the exponential map for $g$ at $\xia$ (for the precise expression of $\Via$ see \eqref{Via} below). The sequences $(\xia)_\alpha$ and $(\mia)_\alpha$, $1 \le i \le k$, moreover satisfy (see e.g. Bahri-Coron \cite{BahriCoron} and Solimini \cite{Solimini})
\ben \label{structure}
\mia \to 0 \quad \textrm{ and } \quad \frac{\mia}{\mja} + \frac{\mja}{\mia} + \frac{d_g(\xia, \xja)^2}{\mia \mja} \to + \infty  \quad  \textrm{ for all  } i \neq j
\een
as $\alpha \to + \infty$, we have $\ua \rightharpoonup u_0$ in $H^1(M)$ and the energy of $\ua$ is quantified:
\[ \Vert \ua \Vert_{H^1(M)}^2 =  \Vert u_0 \Vert_{H^1(M)}^2 + \sum_{i=1}^k   \Vert \nabla V_i \Vert_{L^2(\R^n)}^2 + o(1) \]
as $\alpha \to + \infty$. When all the functions $\ua$ are assumed to be \emph{positive} the profiles $\Via$ are equal to the so-called positive standard bubbles 
\[
 \Bia(x) = \frac{\mia^\pui}{\big(\mia^2 + \frac{d_{g}(\xia,x)^2}{n(n-2)} \big)^{\pui}}, \quad x \in M.
\]
The $\Bia$'s carry the same energy up to an error term and, when $\ua \ge 0$, the energy of $\ua$ is quantified as 
 \[ \Vert \ua \Vert_{H^1(M)}^2 =  \Vert u_0 \Vert_{H^1(M)}^2 + k E_n  + o(1), \]
where $E_n$ is a dimensional constant (see Section \ref{introyamabe} below for more details). In the sign-changing regime the situation is more intricate as we do not have one sole energy level for the $V_i$'s (see e.g. Ding \cite{Ding}). We let, for all $x \in M$, 
\be
B_{0,\alpha}(x)  = \left \{ \bal & 0 & \textrm{ if } u_0 \equiv 0 \textrm{ and } \ker(\triangle_g +h) = \{0\} \\ & 1 & \textrm{ otherwise} \eal \right.  ,\\
\ee
where we keep the notations of \eqref{struwe}.

\medskip

An important issue is to know whether the description \eqref{struwe} holds true in the $C^0$ space, and thus whether or not pointwise estimates on $(\ua)_\alpha$ inherited from \eqref{struwe} hold. This has been known to hold true for positive solutions of \eqref{eqha} for almost two decades and has had spectacular applications, but no analogous results were obtained in the sign-changing case. In this work we answer this question positively when $\ua$ is allowed to change sign. Our main result is as follows : 
\begin{theo} \label{theorieC0}
Let $(M,g)$ be a closed Riemannian manifold of dimension $n \ge 3$, let $(h_\alpha)_\alpha$ be a sequence of functions converging in $L^\infty(M)$ to some function $h$ in $M$ and let $(\ua)_\alpha$ be a sequence of solutions of \eqref{eqha} that satisfies \eqref{blowup}. Let $u_0$ and $V_{1,\alpha}, \dots, V_{k,\alpha}$ be as in the Struwe decomposition \eqref{struwe} of $\ua$, where the $\Via$'s are given by \eqref{Via} below. Then, up to passing to a subsequence for $(\ua)_\alpha$, we have
\[
\frac{\ua - u_0 - \sum_{i=1}^k \Via}{\sum_{i=0}^k \Bia } \to 0 \]
in $L^\infty(M)$ as $\alpha \to + \infty$.
\end{theo}
The statement of Theorem \ref{theorieC0} reformulates as follows: if $(\ua)_\alpha$ is a sequence of solutions of \eqref{eqha} that blows-up with finite-energy then, up to passing to a subsequence for $(\ua)_\alpha$, for any sequence $(\xa)_\alpha$ of points of $M$
\[ \ua(\xa) = u_0(\xa) + \sum_{i=1}^k \Via(\xa) + o \Big( \sum_{i=0}^k \Bia(\xa) \Big)\]
holds as $\alpha \to + \infty$. In other words, the difference between $\ua$, the weak limit and the whole set of (possibly sign-changing) bubbles in Struwe's decomposition is globally pointwise small with respect to the sum of the positive bubbles. In Theorem \ref{theorieC0} no assumption on $h$ is required: in particular one does not need to assume that $\triangle_g + h$ has no kernel. No assumption on the limiting solutions $V_i$, $1\le i \le k$, of the Yamabe equation in $\R^n$ is made either.

\medskip

Theorem \ref{theorieC0} is the first comprehensive result of this type for finite-energy \emph{sign-changing} solutions of \eqref{eqha} in dimensions $n \ge 3$. For \emph{positive} solutions, however, the subject has a rich history, and pointwise blow-up estimates have been known for decades to be a fundamental technical tool to prove both existence results and \emph{a priori} bounds. 

If no bound on the energy is assumed, sharp blow-up estimates for positive solutions of \eqref{eqha} have been used to prove compactness and stability results. A remarkable success is the compactness of positive solutions of the Yamabe equation on manifolds which are not conformally diffeomorphic to the standard sphere, that corresponds to the case $h_\alpha \equiv \frac{n-2}{4(n-1)} S_g$, where $S_g$ is the scalar curvature of $(M,g)$: see Li-Zhu \cite{LiZhu}, Druet \cite{DruetJDG}, Marques \cite{Marques}, Li-Zhang \cite{LiZhang}, Khuri-Marques-Schoen \cite{KhuMaSc} for compactness statements until dimension $24$ and Brendle \cite{Brendle} and Marques-Brendle \cite{BrendleMarques} for counter-examples when $n \ge 25$. We also mention the stability result for positive solutions of \eqref{eqha} of Druet \cite{DruetYlowdim} when $n \ge 3$ and $h < \frac{n-2}{4(n-1)} S_g$ in $M$ (for the definition of what we mean by ``stability'' here see Hebey \cite{HebeyZLAM}). Global degree arguments based on a sharp pointwise blow-up description of positive solutions have also been used to prove involved existence and compactness results for the Nirenberg problem on the sphere, see e.g. the celebrated results of Chen-Lin \cite{ChenLinMP} and Li \cite{LiSn2}. It is worth mentioning that in most of these results one needs to prove, in the course of the analysis, that possible concentration points are isolated, thus reducing the problem to understanding isolated one-bubble blow-up. 

This is however specific to the framework of the aforementioned results: even under the assumption \eqref{blowup} of finite-energy blow-up, and as the counter-examples of Robert-V\'etois \cite{RobertVetois2} show, concentration points for sequences $(\ua)_\alpha$ of positive solutions of \eqref{eqha} need not be isolated. One then needs a pointwise description like that of Theorem \ref{theorieC0} to describe the interactions between non-isolated bubbling profiles. Theorem \ref{theorieC0} for \emph{positive} solutions of \eqref{eqha} satisfying \eqref{blowup} was first proven in Druet-Hebey-Robert \cite{DruetHebeyRobert}, and the proof was subsequently shortened in Druet-Hebey \cite{DruetHebey2} (see also Hebey \cite{HebeyZLAM}). It was then successfully applied to obtain \emph{a priori} bounds for finite-energy solutions of \eqref{eqha}. Druet \cite{DruetJDG} proved for instance that, if $n=4,5$ or $n \ge 7$, $h_\alpha \to h$ in $C^1(M)$ and $h$ satisfies $h(x) \neq \frac{n-2}{4(n-1)} S_g(x)$ for all $x \in M$, every positive finite-energy sequence of solutions of \eqref{eqha} is uniformly bounded in $C^2(M)$ (for details on the six-dimensional case see Druet-Hebey \cite{DruetHebey2}).

 The case $h_\alpha \equiv h >  \frac{n-2}{4(n-1)} S_g$ is worth mentioning: no finite-energy positive blowing-up solutions exist by Druet's \cite{DruetJDG} result in this case but sequences of positive solutions of \eqref{eqha} whose energy goes to infinity have been constructed on $\mathbb{S}^n$ by Chen-Wei-Yan \cite{ChenWeiYan} for $n \ge 5$ (see also V\'etois-Wang \cite{VetoisWang} for $n=4$). This intermediate case shows the relevance of developing analytical tools like Theorem \ref{theorieC0} to specifically deal with the finite-energy case. Let us also mention that analogous results to Theorem \ref{theorieC0} exist in dimension $2$ for positive solutions of Liouville-type and Moser-Trudinger-type equations, and have been used to obtain existence and quantification results. See for instance Chen-Lin \cite{ChenLin}, Li-Shafrir \cite{LiShafrir}, Druet-Thizy \cite{DruetThizy} and the references therein.

\medskip

For \emph{sign-changing} solutions of \eqref{eqha} the situation is much less understood. If no bound on the energy is assumed the situation is dimmer than for positive solutions. As shown by Premoselli-V\'etois \cite{PremoselliVetois}, one can only expect compactness in dimensions $n \ge 7$, when $h < \frac{n-2}{4(n-1)}S_g$ and when the negative part of the sequence $(\ua)_\alpha$ is a priori bounded. The counterpart of Druet \cite{DruetYlowdim} fails for sign-changing solutions in all the other cases, as proven by Ding \cite{Ding} for the Yamabe equation, Premoselli-V\'etois \cite{PremoselliVetois} and V\'etois \cite{Vetois}. The situation for sign-changing finite-energy blow-up is also considerably more involved than in the positive case: finite-energy sign-changing solutions of the Yamabe equation in $\R^n$ are not classified, hence quantification of the energy is lost, and pointwise interactions between bubbles are weaker due to stronger decay at infinity. Few sign-changing finite-energy compactness results are known and this was mainly due, up to now, to the lack of precise pointwise blow-up estimates as those of Theorem \ref{theorieC0}. We mention for instance the finite-energy compactness result of V\'etois \cite{Vetois}  when $n\ge 7$, $h < \frac{n-2}{4(n-1)}$ and $(M,g)$ is locally conformally flat, that relies on arguments from Devillanova-Solimini \cite{DevillanovaSolimini} in the Euclidean case. We also mention Ghoussoub-Mazumdar-Robert \cite{GhoussoubMazumdarRobert}, where the authors obtain sharp pointwise blow-up estimates for sign-changing solutions of Hardy-Schr\"odinger equations with boundary singularity, a case in which the bubbles are centered at the singularity. 

Theorem \ref{theorieC0} paves the way for more sophisticated sign-changing compactness results. Similar techniques have recently been applied to obtaining quantitative stability results for Struwe's decomposition for positive Palais-Smale sequences, see Figalli-Glaudo \cite{FigalliGlaudo} and Deng-Sun-Wei \cite{DengSunWei} in the positive case. It would be interesting to see if the approach of Theorem \ref{theorieC0} can be applied to obtain analogous quantitative stability results in the sign-changing case.

\subsection{Strategy of proof and organisation of the paper}

To prove Theorem \ref{theorieC0} we develop a new strategy of proof, different from those in Druet-Hebey-Robert \cite{DruetHebeyRobert}, Druet-Hebey \cite{DruetHebey2} and Ghoussoub-Mazumdar-Robert \cite{GhoussoubMazumdarRobert}, that allows us to overcome the technical challenges inherent to the sign-changing setting, such as the failure of Harnack inequalities and maximum principles or the presence of the kernel of the limiting operator $\triangle_g +h$. Our proof combines a bubble-tree analysis with a finite-dimensional reduction. The latter is in the spirit of Ambrosetti-Malchiodi \cite{AmbrosettiMalchiodi2, AmbrosettiMalchiodi}, Bahri \cite{Bahri}, Rey \cite{Rey}, Rey-Wei \cite{ReyWei} or Wei \cite{WeiGM}.

The paper is organised as follows. In Section \ref{introyamabe} we recall a few results on sign-changing solutions of the Yamabe equation in $\R^n$ and introduce the bubbling profiles we will consider. Section \ref{linear} is the core of the analysis of the paper. We linearize equation \eqref{eqha} at an approximate solution that is close to $u_0 + \sum_{i=1}^k \Via$ (more general configurations have to be taken into account) and we prove Theorem \ref{proplin} below. Theorem \ref{proplin} essentially says that solutions of the linearized equation inherit explicit pointwise estimates whose precision depends on the \emph{pointwise} precision of the error of the approximate solution $u_0 + \sum_{i=1}^k \Via$.  It can be seen as a linearised version of Theorem \ref{theorieC0}. The proof of Theorem \ref{proplin} goes through an \emph{a priori} bubble-tree analysis, where sharp estimates are obtained by an inductive bootstrap procedure in the bubble-tree. It requires a full \emph{a priori} analysis and is not a simple invertibility result in weighted spaces.  In Section \ref{nonlinear} we first perform a nonlinear perturbative argument (Proposition \ref{nonlin2}) and prove that \eqref{eqha} can be uniquely solved, up to kernel elements, by a function that satisfies explicit pointwise estimates; the analysis takes place in strong spaces and heavily relies on Theorem \ref{proplin}.  As a consequence we prove a slightly more general result than Theorem \ref{theorieC0}, Theorem \ref{theopsC0} below, that shows that approximate solutions of \eqref{eqha} obtained by Lyapunov-Schmidt-like techniques satisfy the same pointwise bounds as in Theorem \ref{theorieC0}. This is a result which if of interest in itself. Finally the Appendix collects a few technical results used throughout the paper. 

\medskip

The method of proof that we develop in Sections \ref{linear} and \ref{nonlinear} is particularly robust and applies to more general equations than \eqref{eqha}. The invertibility result of Section \ref{linear}, Theorem \ref{proplin}, allows us to bypass most of the technical difficulties that arise when proving blow-up pointwise estimates for critical nonlinear equations, since it only requires a representation formula for a linear problem. The approach developed in this paper adapts for instance to higher-order critical nonlinear equations involving linear operators that look like $(\triangle_g)^p$ for some integer $p \ge 1$ at first-order. It also seems particularly suited to deal with highly nonlinear equations.

\medskip
\noindent {\bf Acknowledgments:} The author was supported by a FNRS CdR grant J.0135.19, by the Fonds Th\'elam and by an ARC Avanc\'e 2020 grant.

\section{Preliminary material} \label{introyamabe} 

\subsection{Finite-energy nodal solutions of the Yamabe equation in $\R^n$.}

Let $n \ge 3$ and denote by $D^{1,2}(\R^n)$ the completion of $C^\infty_c(\R^n)$ with respect to the norm $u \mapsto \int_{\R^n} |\nabla u|^2 dx$, endowed with the associated scalar product. Let $V \in D^{1,2}(\R^n)$ be a solution of the nodal Yamabe equation in $\R^n$:
\ben \label{yamabe}
\triangle_{\xi} V = |V|^{2^*-2} V,
\een
where $\xi$ is the Euclidean metric and $\triangle_\xi = - \sum_{i=1}^n \partial_i^2$. Unless specified otherwise $V$ is allowed to change sign. If $V$ is positive it is equal by Obata \cite{Obata}, up to translations and rescalings, to the so-called positive standard bubble
\ben \label{B0}
 B_0(x) = \( 1 + \frac{|x|^2}{n(n-2)} \)^{- \pui},
 \een
and is in particular an extremal for the euclidean sharp Sobolev inequality. By Caffarelli-Gidas-Spruck \cite{CaGiSp} this result remains true for all positive solutions of \eqref{yamabe}, without the assumption that they belong to $D^{1,2}(\R^n)$. The positive standard bubbles satisfy 
\[ \int_{\R^n} |\nabla B_0|^2dx = K_n^{-n}, \]
where $K_n$, an explicit dimensional constant, is the sharp constant for the embedding of $D^{1,2}(\R^n)$ into $L^{2^*}(\R^n)$. By contrast, any sign-changing solution $V$ of \eqref{yamabe} satisfies 
\[ \int_{\R^n} |\nabla V|^2 dx > 2 K_n^{-n} \]
as is easily seen by splitting $V$ into its positive and negative part. A result of Weth \cite{Weth} shows the existence of a dimensional constant $\ve_n >0$ so that every sign-changing solution of \eqref{yamabe} satisfies  $\int_{\R^n} |\nabla V|^2 dx \ge 2 K_n^{-n} + \ve_n$. 

\medskip

We describe in this subsection the behavior at infinity for solutions of \eqref{yamabe} and of the linearized equation. We first recall the following simple result: 

\begin{lemme}
Let $V \in D^{1,2}(\R^n)$ be a solution of \eqref{yamabe}. 
\begin{itemize}
\item There exists a constant $C = C(n,V)$ such that, for all $x \in \R^n$, 
\ben \label{estV}
|V(x)| + (1+|x|)|\nabla V(x)| \le C (1 + |x|)^{2-n}.
\een
\item There exists $\lambda \in \R$ such that 
\ben \label{asymptoV}
V(x) = \frac{\lambda}{|x|^{n-2}} + O(|x|^{1-n}) \quad \textrm{ as } |x| \to + \infty,
\een
and it satisfies
\ben \label{caracl}
(n-2) \omega_{n-1} \lambda = \int_{\R^n} |V|^{2^*-2}V dx, 
\een
\end{itemize}
where $\omega_{n-1}$ is the area of the standard sphere $\mathbb{S}^{n-1}$.
\end{lemme}
If $V$ has constant sign $\lambda = \pm (n(n-2))^{\pui}$ by \eqref{B0}.

\begin{proof}
First, by the regularity theory for critical elliptic equations (see e.g. Trudinger \cite{Trudinger} or also Hebey \cite{HebeyZLAM}, Theorem $2.15$), $V$ is smooth in $\R^n$, and \eqref{estV} is true for $x \in \overline{B(0,1)}$. For $x \neq 0$ let $V^*(x) = |x|^{2-n} V\big( \frac{x}{|x|^2}\big)$ be the Kelvin transform of $V$. It is a well-known fact that $V^*$ still satisfies \eqref{yamabe} on $\R^n \backslash \{0\}$. The Kelvin transform leaves the $L^{2^*}(\R^n)$ norm invariant and hence $V^* \in L^{2^*}(\R^n)$. Standard removal of singularities results (see e.g. the arguments in Caffarelli-Gidas-Spruck \cite{CaGiSp}, Lemma $2.1$) then show that $V^*$ is a distributional solution of \eqref{yamabe} in the whole $\R^n$ and hence belongs to $D^{1,2}(\R^n)$. Regularity theory shows again that $V^*$ smoothly extends at the origin. Hence there exists a constant $C = C(n,V) >0$ such that 
\[ \Vert V^* \Vert_{L^\infty(B(0,1))} + \Vert \nabla V^* \Vert_{L^\infty(B(0,1))} \le C, \]
which, by definition of $V^*$, proves \eqref{estV} for $x \in \R^n \backslash \overline{B(0,1)}$.   

At the origin $V^*$ moreover expands as $V^*(z) = V^*(0) + O(|z|)$. Letting $z = \frac{x}{|x|^2}$ yields \eqref{asymptoV} with $\lambda = V^*(0)$ when $|x| \to + \infty$. To prove \eqref{caracl} we write a representation formula for $V$ in $\R^n$. Let $R >0$ be fixed and choose $|x|\ge 2R$. With \eqref{estV} we can write that, for any $x \in \R^n$,
\[ \bal 
V(x) &= \int_{\R^n} \frac{1}{(n-2)\omega_{n-1}} |x-y|^{2-n} |V(y)|^{2^*-2} V(y) dy \\
& = \int_{B(0,R)} \frac{1}{(n-2)\omega_{n-1}} |x-y|^{2-n} |V(y)|^{2^*-2}V(y) dy  + O(\frac{1}{R^2} (1+|x|)^{2-n}).
\eal  \]
Multiplying by $|x|^{n-2}$ both sides of the equality and letting first $|x| \to + \infty$ and then $R \to + \infty$ proves \eqref{caracl}.
\end{proof}

\medskip

If $V\in D^{1,2}(\R^n)$ solves \eqref{yamabe} we let
\ben \label{defKV}
K_V = \Big \{ h \in D^{1,2}(\R^n), \quad \triangle_\xi h = (2^*-1) |V|^{2^*-2}h \Big \}
\een 
be the set of solutions of the linearised equation at $V$. A simple integration by parts yields $V \in K_V^{\perp}$. The following result proves decay at infinity for the elements of $K_V$:
\begin{lemme}
Let $Z \in K_V$. There exists $C = C(n,V,Z)$ such that 
\ben \label{estZ}
|Z(x)| + (1+|x|)|\nabla Z(x)| \le C (1 + |x|)^{2-n}
\een
for all $x \in \R^n$.
\end{lemme}

\begin{proof}
Let, for $x \in \R^n$, $U(x)^{\frac{4}{n-2}} = 4 (1+|x|^2)^{-2}$. Let $\pi: \mathbb{S}^n \backslash \{N\} \to \R^n$ be the stereographic projection from the North pole and let, for $\vp \in D^{1,2}(\R^n)$ and for $y \in \mathbb{S}^n$, $\tilde{\vp}(y) = \frac{\vp}{U}(\pi(y))$. The conformal invariance of the conformal laplacian shows that, for any $y \in \mathbb{S}^n$, $\tilde{\vp}$ satisfies
\ben \label{confinv} \Big( \triangle_{g_0} + \frac{n(n-2)}{4} \Big)\tilde{\vp}(y) = U(\pi(y))^{1 - 2^*} \triangle_{\xi} \vp (\pi(y)) \een
where $g_0$ is the round metric on $\mathbb{S}^n$. Recall that $g_0$ satisfies $g_0 = \pi^*\big( U^{\frac{4}{n-2}} \xi)$. Let $Z \in K_V$ and consider $\tilde{Z}$. The assumption $Z \in D^{1,2}(\R^n)$ translates, by \eqref{confinv}, into
\[ \int_{\mathbb{S}^n} |\nabla \tilde{Z}|_{g_0}^2 + \frac{n(n-2)}{4} |\tilde Z|^2 dv_{g_0} = \int_{\R^n} |\nabla Z|^2 dx < + \infty, \]
so that $\tilde{Z} \in H^1(\mathbb{S}^n)$. And by \eqref{defKV} $\tilde{Z}$ satisfies
\ben \label{confinv2}  \triangle_{g_0} \tilde{Z} + b \tilde{Z}=0 \quad \textrm{ in } \mathbb{S}^n, \een
where we have let
\[ b = \frac{n(n-2)}{4} - (2^*-1) \(\frac{|V(\pi(\cdot))|}{U(\pi(\cdot))} \)^{2^*-2}. \]
Elements of $K_V$ are actually in one-to-one correspondence with solutions $\tilde{Z} \in H^1(\mathbb{S}^n)$ of \eqref{confinv2} via the transformation $Z \to \tilde{Z}$. Estimate \eqref{estV} implies $\Vert b \Vert_{L^\infty(\mathbb{S}^n)} \le C(n,V)$. Standard elliptic theory for \eqref{confinv2} then shows that $\Vert \tilde{Z} \Vert_{C^1(\mathbb{S}^n)} \le C(n,V)$. Since, for all $x \in \R^n$, $Z(x) = U(x)\tilde{Z}(\pi^{-1}(x))$, this proves \eqref{estZ}.
\end{proof}

In the following we will note $n_V  = \dim K_V$ and we will denote by $\Pi_{K_V}$ the orthogonal projection on $K_V$ in $D^{1,2}(\R^n)$. We always have $n_V >0$, and this follows e.g. from the invariance of solutions of \eqref{yamabe} by scaling. In general little more can be said on the dimension of $K_V$ as in \eqref{defKV}. It is known that for \emph{positive} solutions of \eqref{yamabe} $n_V = n+1$ holds (see Bianchi-Egnell \cite{BianchiEgnell}) and that positive solutions are non-degenerate in the sense of Duyckaerts-Kenig-Merle \cite{DuyckaertsKenigMerle}. Constructive examples of non-degenerate nodal solutions of \eqref{yamabe} whose value of $n_V$ is known  have been obtained in Del Pino-Musso-Pacard-Pistoia \cite{DelPinoMussoPacardPistoia1, DelPinoMussoPacardPistoia2}, Medina-Musso \cite{MedinaMusso}, Medina-Musso-Wei \cite{MedinaMussoWei} and Musso-Wei \cite{MussoWei}. No example of \emph{degenerate} solution of \eqref{yamabe} is known to this day, but the existence of such solutions has not been ruled out.

\subsection{Riemannian sign-changing bubbles} \label{sectionbulles}

We let until the end of this section $(M,g)$ be a closed manifold (that is, compact without boundary) of dimension $n \ge 3$ and let $h \in L^\infty(M)$. Let $(\xa)_\alpha$ be a sequence of points in $M$ and $(\ma)_{\alpha}$ be a sequence of positive numbers and let $V \in D^{1,2}(\R^n)$ be a solution of \eqref{yamabe}. Let $\lambda \in \R$ be given by \eqref{asymptoV}. Let $\chi \in C^\infty_c(\R)$ be such that $\chi \equiv 1$ in $[0,\frac{i_g}{2}]$ and $\chi \equiv 0$ in $\R \backslash [0,i_g]$, where $i_{g}$ denotes the injectivity radius of $g$. If $\triangle_g + h$ has no kernel let $G_h$ denote the Green's function of $\triangle_g +h$ (see Aubin \cite{Aub} or Robert \cite{RobDirichlet}) and define 
\[ F(\xa, x) =  (n-2)\omega_{n-1} d_g(\xa, x)^{n-2} G_h(\xa, x). \]
We define a riemannian sign-changing bubble $V^{\ma,\xa}$ on $M$, modeled on $V$, as follows: 
\ben \label{Vxa}
\bal
\bullet \textrm{ If } \ker(\triangle_g + h) &= \{0\}: \\
 V^{\ma, \xa}(x) & =  \chi \big( d_{g}(\xa, x) \big)F(\xa, x) \ma^{-\pui}V \Big( \frac{1}{\ma} \exp_{\xa}^{-1}(x) \Big)  \\
 & + \big(1 -  \chi \big( d_{g}(\xa, x) \big) \big) (n-2)\omega_{n-1} \lambda \ma^{\pui} G_h(\xa, x) \\
\bullet \textrm{ If } \ker(\triangle_g + h) &\neq \{0\}: \\
V^{\ma, \xa}(x) & = \chi \big( d_{g}(\xa, x) \big) \ma^{-\pui} V \Big( \frac{1}{\ma} \exp_{\xa}^{-1}(x) \Big) 
\eal
\een
for all $x \in M$, where $\exp_{\xa}$ denotes the exponential map at $\xa$ for $g$. We also let, for $x \in M$,
\ben \label{Ba}
B_{\ma,\xa}(x) = \frac{\ma^\pui}{\big(\ma^2 + \frac{d_{g}(\xa,x)^2}{n(n-2)} \big)^{\pui}}
\een
be the Riemannian positive standard bubble centered at $\xa$. It is easily seen with \eqref{estV} and \eqref{Vxa} that we have 
\ben \label{contVa}
 |V^{\ma,\xa}(x)| + (\ma + d_g(\xa, x)) |\nabla V^{\ma,\xa}(x)| \le C B_{\ma,\xa}(x) 
 \een
for all $x \in M$, where $C$ only depends on $n,(M,g)$ and $V$. Let $(h_\alpha)_{\alpha}$ be a sequence of functions in $M$ that converges towards $h$ in $L^\infty(M)$. Straightforward computations using \eqref{Vxa} show that there exists $C=C(n,g,V,h) >0$ such that 
\ben \label{errVxa}
\left| \triangle_g V^{\ma,\xa} + h_\alpha V^{\ma,\xa} - |V^{\ma,\xa}|^{2^*-2} V^{\ma,\xa}  \right| \le C B_{\ma,\xa}
\een
everywhere in $M$ (see e.g. Esposito-Pistoia-V\'etois \cite{EspositoPistoiaVetois}, proposition 2.2). Let $k \ge 1$ be an integer and let $(\xia)_{\alpha}$, $1 \le i \le k$, be $k$ sequence of points of $M$ and  $(\mia)_\alpha$, $1 \le i \le k$ be  $k$ sequences of positive real numbers. We will let throughout this paper:
\ben \label{Via}
\Via = V_i^{\mia, \xia}, 
\een
where $V_i^{\mia, \xia}$ is given by \eqref{Vxa}. 

\subsection{Kernel elements} We endow $H^1(M)$ with the scalar product
\[ \langle u, v \rangle_{H^1(M)} = \int_M \Big( \langle \nabla u, \nabla v \rangle + u v \Big) dv_g. \]
Let $h \in L^\infty(M)$ and let $u_0 \in C^{1,\beta}(M)$, $0 < \beta < 1$, be a solution of
\ben  \label{limiteq}
\triangle_g u_0 + h u_0 = |u_0|^{2^*-2} u_0 .
\een
Let $k \ge 1$, let $V_1, \dots, V_k$ be solutions of \eqref{yamabe} and let $(\xia)_{\alpha}$ and $(\mia)_\alpha$, $1 \le i \le k$, be $k$ sequences respectively of points in $M$ and of positive real numbers with $\mia \to 0$ as $\alpha \to + \infty$ satisfying \eqref{structure}. 

\medskip

For each $1 \le i \le k$ let $n_i = \textrm{dim} K_{V_i}$ where $K_{V_i}$ is as in \eqref{defKV} and let $Z_{i,1}, \dots, Z_{i, n_i}$ be an orthonormal basis of $K_{V_i}$. As before let $\chi \in C^\infty_c(\R)$ satisfy $\chi \equiv 1$ in $[0,\frac{i_g}{2}]$ and $\chi \equiv 0$ in $\R \backslash [0,i_g]$. For any $\alpha \ge 1$, $1 \le i \le k$ and $1 \le j \le n_i$ we let, for $x \in M$,
\ben \label{defZij}
 Z_{i,j,\alpha} (x)=   \chi ( d_g(\xia, x) ) \mia^{-\pui} Z_{i,j}\Big( \frac{1}{\mia} \exp_{\xia}^{-1}(x) \Big)  \een
and 
\[ K_{i,\alpha} = \textrm{Span}\big \{  Z_{i,j,\alpha} , 1 \le j \le n_i\big \} \] 
which is a subset of $H^1(M)$. We also let 
 \ben \label{defK0}
  K_0 = \Big \{ v \in H^1(M), \quad \triangle_g v + h v = (2^*-1)|u_0|^{2^*-2} v \Big \}, 
  \een
 and we let $(Z_{0,j})_{1 \le j \le n_0}$ be an orthonormal basis of $K_0$ for the $H^1(M)$ scalar product. In case $u_0 \equiv 0$, $K_0$ coincides with the kernel of $\triangle_g + h$. Let finally, for $\alpha \ge 1$,
 \ben \label{defKa}
 K_\alpha = K_0 \oplus \bigoplus_{i=1}^k K_{i,\alpha}.
 \een
 If follows from \eqref{defZij} that $Z_{i,j,\alpha} \rightharpoonup 0$ in $H^1(M)$ as $\alpha \to + \infty$. Using \eqref{orthoZij} below we then obtain that, up to passing to a subsequence for $(\xia)_\alpha$ and $(\mia)_\alpha$, $1 \le i \le k$, $K_\alpha$ is a linear subset of $H^1(M)$ of dimension $\sum_{i=0}^k n_i$. Note that $n_0 = \dim K_0$ may be equal to $0$, in which case our convention is that the condition $1 \le j \le n_0$ is empty.

\section{Linear blow-up analysis} \label{linear}

In this section we linearise equation \eqref{eqha} around the first-order terms $u_0 + \sum_{i=1}^k \Via$ in the Struwe decomposition \eqref{struwe}. The main result of this section, Theorem \ref{proplin}, establishes \emph{a priori} pointwise estimates for solutions of the linearised equation. This result holds true for any possible bubbling configuration.

\subsection{Linear blow-up analysis}

Let
\ben \label{defE}
\mathcal{E} = \big \{ (\vea)_{\alpha \ge 1} \textrm{ such that } 0 < \vea \le 1  \textrm{ for all } \alpha \ge 1 \textrm{ and } \lim_{\alpha \to + \infty} \vea = 0  \big\}.
\een
Let $(h_\alpha)_\alpha$ be a sequence of functions converging towards $h$ in $L^\infty(M)$ as $\alpha \to + \infty$. Let $k \ge 1$ be an integer, let $u_0 \in C^{1,\beta}(M)$ for all $0 < \beta < 1$ (possibly equal to zero) solve \eqref{limiteq}, let $V_1, \dots, V_k $ in $D^{1,2}(\R^n)$ be $k$ solutions of \eqref{yamabe} and let $(\xia)_{\alpha}$ and $(\mia)_\alpha$, $1 \le i \le k$ be $k$ sequences of points in $M$ and positive real numbers satisfying \eqref{structure}. For $i \in \{1, \dots, k \}$ and $x \in M$ we define
\ben \label{defti}
\tia(x) = \mia + d_g(\xia, x),
\een
\ben \label{defTi}
\Theta_{i, \alpha}(x) = \left \{ \begin{aligned} & \tia(x) & \textrm{ if } n =3 \\ & \tia^2(x) \big| \ln \tia(x) \big| & \textrm{ if } n = 4 \\ & \tia^2(x) & \textrm{ if } n \ge 5\end{aligned} \right. 
\een
and
\ben \label{defBia}
\begin{aligned}
B_{0,\alpha}(x) & = \left \{ \bal & 0 & \textrm{ if } u_0 \equiv 0 \textrm{ and } \ker(\triangle_g +h) = \{0\} \\ & 1 & \textrm{ otherwise} \eal \right.  \\
\Bia(x) & = B_{\mia, \xia}(x) \textrm{ for } 1 \le i \le n,
\end{aligned} 
\een
where $B_{\mia, \xia}$ is as in \eqref{Ba}. By definition $B_{0,\alpha} \equiv 0 \iff u_0 \equiv 0 $ and $K_0 = \{0\}$, where $K_0$ is as in \eqref{defK0}. We  will let, throughout this section and the next one, $(\mWa)_\alpha$ be a sequence of continuous functions in $M$ satisfying 
\ben \label{defWa}
\mathcal{W}_\alpha = u_{0,\alpha} + \sum_{i=1}^k W_{i,\alpha},
\een
where $u_{0,\alpha}  \equiv 0 $ if  $u_0 \equiv 0$ and  $\ker(\triangle_g +h) = \{0\}$ and $ \Vert u_{0,\alpha} - u_0 \Vert_{L^\infty(M)} \le 1$ and $ \Vert \uoa - u_0 \Vert_{L^\infty(M)} \to 0 $ otherwise, and where, for $1 \le i \le k$,
 \[ \left \{ \begin{aligned}
& |W_{i,\alpha}(x)|  \le C \Bia(x) \quad \textrm{ for all } x \in M \\
& |\nabla W_{i,\alpha}(x)|  \le C \tia(x)^{-1}\Bia(x) \quad \textrm{ for all } x \in M \\ 
&  \mia^{\pui} W_{i,\alpha} \big( \exp_{\xia}(\mia y) \big)  \to V_i(y) \quad \textrm{ in } \quad C^1_{loc}(\R^n) 
 \end{aligned} \right. \]
 as $\alpha \to + \infty$, for some $C >0$ that only depends on $n,(M,g), V_1, \dots V_k, u_0$ and $h$.

\medskip

We will repeatedly make use of the following notations:
\begin{itemize}
\item If $(a_\alpha)_\alpha$ and $(b_\alpha)_\alpha$ are sequence of real numbers (or functions) with $b_\alpha > 0$ we will write 
\ben \label{ineqq}
 |a_\alpha| \lesssim b_{\alpha} \textrm{ or, equivalently, } a_\alpha = O(b_\alpha)
 \een
if there exists a positive constant $C$ that only depends on $n,(M,g),k$, $V_1, \dots, V_k, u_0$ and $h$ such that $|a_\alpha| \le C b_\alpha$. In particular such a constant $C$ is independent of $\alpha$ and of the choice of $(\mia)_\alpha$, $(\xia)_\alpha$, $1 \le i \le k$. For sequences of real numbers we will denote  $a_\alpha = o(b_\alpha)$ if $\frac{a_\alpha}{b_\alpha} \to 0$ as $\alpha \to + \infty$. 
\item We will denote by $\sa$ any sequence of positive numbers converging to $0$ that only depends on the sequences $(\Vert h_\alpha -h \Vert_{L^\infty(M)})_\alpha$, $(\xia)_\alpha$ and $(\mia)_\alpha$, $1 \le i \le k$. Schematically:
\ben \label{defsigmaa}
\sa = \sa \Big(\Vert h_\alpha -h \Vert_{L^\infty(M)}, (\xia)_{1 \le i \le k }, (\mia)_{1 \le i \le k} \Big), \quad \lim_{\alpha \to + \infty} \sa = 0. 
\een
The notation $\sa$ might denote, from one line to the other, different sequences still satisfying \eqref{defsigmaa}. In some cases, as for instance in the proof of Proposition \ref{recurrence} below, several sequences of this type will have to be distinguished, and they will be numbered as $\sa^1, \sa^2, \dots$. 
\end{itemize}

\medskip

 If $i,j,p,q$ are integers satisfying  $0 \le i,j \le k$, $1 \le p \le n_i$ and $1 \le q \le n_j$ we have for instance
 \ben \label{orthoZij}
\langle Z_{i,p,\alpha}, Z_{j,q,\alpha} \rangle_{H^1(M)} = \delta_{ij} \delta_{pq} + O(\sa),
\een
and this follows from \eqref{defZij}. For simplicity in \eqref{orthoZij} we have let, when $i=0$, $Z_{0,j,\alpha} = Z_{0,j}$, $1 \le j \le n_0$, where $Z_{0,j}$ spans $K_0$. Similarly, by \eqref{limiteq} we have
\[ \Vert  \triangle_g u_0 +  h_\alpha u_0 - |u_0|^{2^*-2}u_0 \Vert_{L^\infty(M)} \le \sa. \]
Let now, for all $s \in \R$,
\ben \label{deff}
f(s) = |s|^{2^*-2}s 
\een
and let $(\mWa)_\alpha$ be a sequence of continuous functions satisfying \eqref{defWa}. Our goal is to obtain pointwise estimates on solutions of the linearized equation of  \eqref{eqha} at $\mWa$. We first recall the standard linear theory in $H^1(M)$. For $\vp \in K_\alpha^{\perp}$, where $K_\alpha$ is as in \eqref{defKa}, let $L_\alpha: K_\alpha^{\perp} \to K_\alpha^{\perp}$ be given by
\[ L_\alpha \vp = \Pi_{K_\alpha^{\perp}}\Big[ \vp - (\triangle_g + 1)^{-1} \Big (f'( \mWa ) \vp + (1-h_\alpha) \vp \Big) \Big] \]
where $\Pi_{K_\alpha^{\perp}}$ denotes the orthogonal projection on $K_{\alpha}^{\perp}$ for the $H^1(M)$ scalar product. The following is a well-known result:

\begin{prop} \label{propH1}
There exists a constant $C= C(n,g,V_1, \dots, V_k,u_0,h)$ such that for $\alpha$ large enough
\[ \frac{1}{C} \Vert \vp \Vert_{H^1(M)} \le \Vert L_\alpha \vp \Vert_{H^1(M)} \le C \Vert \vp \Vert_{H^1(M)} \quad \textrm{ for all } \vp \in K_\alpha^{\perp}. \]
In particular $ L_\alpha: K_\alpha^{\perp} \to K_\alpha^{\perp}$ is a bicontinuous isomorphism.  
\end{prop}

The proof of Proposition \ref{propH1} follows from standard arguments that can be found e.g. in Robert-V\'etois \cite{RobertVetois}, Esposito-Pistoia-V\'etois \cite{EspositoPistoiaVetois} or Musso-Pistoia \cite{MussoPistoia2}. It shows in particular that for any $\alpha \ge 1$, if $R_\alpha$ is a continuous function in $M$, there exists a unique $\vpa \in K_\alpha^{\perp}$ and unique real numbers $\lija$, $0 \le i \le k$, $1 \le j \le n_i$, such that 
\[ \triangle_g \vpa + h_\alpha \vpa - f'(\mWa) \vpa = R_\alpha +  \sum \limits_{\underset{j =1 .. n_i }{i=0..k}} \lija ( \triangle_g   + 1)Z_{i,j, \alpha}.
\]
By standard elliptic theory $\vpa \in C^1(M)$. The main result of this section describes how $\vpa$ inherits pointwise estimates from $R_\alpha$:

 \begin{theo} \label{proplin}
Let $(M,g)$ be a closed Riemannian manifold of dimension $n \ge 3$, let $(h_\alpha)_\alpha$ be a sequence of functions converging towards $h$ in $L^\infty(M)$, $k \ge 1$ be an integer, $u_0 \in C^{1,\beta}(M)$ for $0 < \beta < 1$ be a solution of \eqref{limiteq}, $V_1, \dots, V_k $ in $D^{1,2}(\R^n)$ be solutions of \eqref{yamabe} and let $(\xia)_{\alpha}$ and $(\mia)_\alpha$, $1 \le i \le k$ be $k$ sequences satisfying \eqref{structure}. Let $(\mWa)_\alpha$ be a sequence of continuous functions in $M$ satisfying \eqref{defWa}. There exist $C_0 = C_0(n,g,V_1, \dots, V_k, u_0,h) >0$ and there exists a sequence $(\sa)_\alpha$ as in \eqref{defsigmaa} such that, up to passing to a subsequence for $(\sa)_\alpha$, the following holds: 

\medskip

Let $(\ea)_\alpha, (\tau_\alpha)_\alpha$ be sequences in $\mathcal{E}$, $(\gamma_\alpha)_\alpha$ be a bounded sequence of real numbers and $(R_\alpha)_\alpha$ be a sequence of continuous functions in $M$ satisfying 
\ben \label{propRa}
|R_\alpha| \le  \tau_\alpha B_{0,\alpha} +  \ea \sum_{i=1}^k \Bia^{2^*-1} + \gamma_\alpha \Big( \sum_{i=1}^k \Bia +   \sum \limits_{\underset{i \neq j}{i,j=0..k}} \Bia^{2^*-2} \Bja\Big)
\een
 in $M$. Let $\vpa \in K_\alpha^{\perp}$ be the unique solution of 
\ben \label{eqvpa} \triangle_g \vpa + h_\alpha \vpa - f'(\mWa) \vpa = R_\alpha +  \sum \limits_{\underset{j =1 .. n_i }{i=0..k}} \lija ( \triangle_g   + 1)Z_{i,j, \alpha} 
 \een
given by Proposition \ref{propH1}. Then $\vpa$ satisfies:
\ben \label{estopt}
|\vpa(x)| \le C_0 \big( \ta + \ea + \sa \gamma_\alpha   \big) \sum_{i=0}^k \Bia(x)  + C_0\gamma_\alpha \sum_{i=1}^k \Tia(x) B_i(x) 
\een
for all $\alpha$ and for all $x \in M$. Also, for all $\alpha$,
\ben \label{vp}
 \sum \limits_{\underset{j =1 .. n_i }{i=0..k}} |\lija|  \le C_0 \big( \ta + \ea +\sa  \gamma_\alpha   \big).
\een
\end{theo}
The quantities $\mathcal{E}, B_{0,\alpha}, \Bia$ and $\Tia$ have been defined in  \eqref{defE}, \eqref{defTi} and \eqref{defBia}. As before we have let $Z_{0,j,\alpha} = Z_{0,j}$, $1 \le j \le n_0$, where $Z_{0,j}$ spans $K_0$. Note that by \eqref{defsigmaa}, passing to a subsequence for $\sa$ amounts to passing to a common subsequence for $(\xia)_{\alpha}$, $(\mia)_\alpha$, $1 \le i \le k$, and $(h_\alpha)_\alpha$.

\medskip

Theorem \ref{proplin} is a generalisation of Proposition \ref{propH1} where the $H^1$ control is improved into effective pointwise estimates. When $\mWa$ is given by $\mWa = u_0 + \sum_{i=1}^k \Via$, \eqref{propRa} is satisfied by $R_\alpha = \triangle_g \mWa + h_\alpha \mWa  - |\mWa|^{2^*-2} \mWa$, and this follows from \eqref{errVxa}. Theorem \ref{proplin} has thus to be interpreted as a linearised version of Theorem \ref{theorieC0}. Here is an alternative way to see \eqref{estopt}: 
assume, for simplicity, that $\triangle_g +h$ has no kernel and let $G_h$ be its Green's function in $M$. Direct computations using \eqref{propRa} show that there exists a sequence $(\sa)_\alpha$ as in \eqref{defsigmaa} such that for all $x \in M$
\[
\begin{aligned}
\int_M G_h(x, \cdot) R_\alpha dv_g &\lesssim \big( \ta + \ea + \sa \gamma_\alpha   \big) \sum_{i=0}^k \Bia(x) \\
& +\gamma_\alpha \sum_{i=1}^k \Tia(x) B_i(x)
\end{aligned}
\]
holds (see \eqref{proplin1} below for more details). Theorem \ref{proplin} thus essentially says that the solution $\vpa$ of \eqref{eqvpa} behaves as if it formally satisfied the equation
\[ \triangle_g u + h u = R_\alpha .\]
A similar interpretation remains true if $\triangle_g +h$ has non-zero kernel, where the projection of $\vpa$ on the kernel of $\triangle_g +h $ has to be taken into account. The method of proof of Theorem \ref{proplin} is flexible and if a more precise estimate than \eqref{propRa} is available on $R_\alpha$ then one recovers a better estimate on $\vpa$ than \eqref{estopt}: roughly speaking one expects $\vpa$ to satisfy the same pointwise estimates as $(\triangle_g +h)^{-1} R_\alpha$. Sharper estimates than \eqref{estopt} have for instance been obtained in Premoselli \cite{Premoselli12} when $n \ge 7$ and  $\mWa$ looks like a tower of positive bubbles (possibly with alternating signs). We also mention Deng-Sun-Wei \cite{DengSunWei} who obtained by different arguments pointwise estimates on solutions of the linearised Yamabe equation in $\R^n$ at specific configurations of sums of finitely many positive bubbles. An important remark here is that our proof of Theorem \ref{proplin} does not require any non-degeneracy assumption in the sense of Duyckaerts-Kenig-Merle \cite{DuyckaertsKenigMerle} on the possibly sign-changing bubbles $V_1, \dots, V_k$: any finite-energy solution of \eqref{yamabe} is allowed in \eqref{defWa}. This is the main reason why we cannot expect, in this paper, to get a better control than \eqref{estopt} on $\vpa$. We refer to Remark \ref{remprecision} below for a more detailed discussion on this point. 

\medskip

  The following four subsections are devoted to the proof of Theorem \ref{proplin}. Subsection \ref{radius} introduces the notion of radius of influence of a nodal bubble and proves pointwise estimates at different scales of the bubble-tree. Subsection \ref{proof1} and \ref{proof2} are devoted to the core of the proof, where improved estimates in the spirit of \eqref{estopt} are proven on the area of influence of each bubble by an inductive argument. Finally subsection \ref{proof3} concludes the proof of Theorem \ref{proplin}.

\subsection{Radius of influence and bubble-tree computations} \label{radius}

Let $k \ge 1$, $V_1, \dots, V_k $ in $D^{1,2}(\R^n)$ be $k$ solutions of \eqref{yamabe}, $u_0$ be a solution of \eqref{limiteq} and $(\xia)_{\alpha}$ and $(\mia)_\alpha$, $1 \le i \le k$ be $k$ sequences of points in $M$ and of positive real numbers satisfying \eqref{structure}. Let $(\mWa)_\alpha$ be as in \eqref{defWa}. Recall that $u_0, V_1, \dots, V_k$ can change sign. In this subsection we introduce the definition of the radius of influence of a nodal bubble and prove a few related results. The exposition is strongly inspired from Druet \cite{DruetJDG} and Druet-Hebey \cite{DruetHebey2} (see also Hebey \cite{HebeyZLAM}, Chapter $8$) where this notion was developed for positive solutions of \eqref{eqha}. Remarkably, it remains extremely useful even to describe the pointwise interactions of nodal bubbles. In the following we will sometimes refer to the sum
\[ \mWa = u_{0,\alpha}  + \sum_{i=1}^k W_{i,\alpha} \]
as a \emph{bubble-tree}. We follow notations \eqref{ineqq} and \eqref{defsigmaa} in this section. 

\medskip

For $ i \in \{1, \dots, k\}$ we define 
\ben \label{defAi}
\mathcal{A}_i = \Big \{ j \in \{1, \dots, k\} \backslash \{i\} \textrm{ such that } \mia  = O(\mu_{j, \alpha}) \Big \} .
\een
These are the bubbles in the family $V_{1,\alpha}, \dots, V_{k,\alpha}$ that concentrate at a slower or comparable speed than the $i$-th bubble -- the so-called \emph{lower} bubbles. If $j \in \mathcal{A}_i$,  we define the radius of interaction of the $i$-th bubble with the lower $j$-th bubble as 
\ben \label{defsij}
s_{i,j,\alpha} = \Big( \frac{\mia}{\mja} \frac{d_g(\xia, \xja)^2}{n(n-2)} + \mia \mja \Big)^{\frac12}. 
\een
The radius of influence of the $i$-th bubble is then defined as follows
\ben \label{defri}
r_{i,\alpha} = \left \{ 
\bal 
& \min \Big( \min_{j \in \mathcal{A}_i} s_{i,j, \alpha}, \frac{i_g}{2} \Big) & \textrm{ if } u_0 \equiv 0 \textrm{ and } \ker(\triangle_g + h) = \{0\} \\
& \min \Big( \min_{j \in \mathcal{A}_i} s_{i,j, \alpha}, \sqrt{\mia}\Big) & \textrm{ otherwise }  \\
\eal \right. .
\een
If $\mathcal{A}_i = \emptyset$ we have in particular 
\[ r_{i,\alpha} = \left \{ 
\bal 
&\frac{i_g}{2} & \textrm{ if } u_0 \equiv 0 \textrm{ and } \ker(\triangle_g + h) = \{0\} \\
& \sqrt{\mia} & \textrm{ otherwise } \\
\eal \right. .  \]
The radius $\ria$ is defined so that $\ria \le \sqrt{\mia}$ whenever $B_{0,\alpha} \equiv 1$ (see \eqref{defBia}), so that $\Bia \gtrsim B_{0,\alpha}$ on $B_g(\xia, \ria)$. By \eqref{structure} we have that $\frac{s_{i,j,\alpha}}{\mia} \to + \infty$ for all $j \in \mathcal{A}_i$ and hence $\frac{r_{i,\alpha}}{\mia} \to + \infty$ as $\alpha \to + \infty$. Define now
\[  \Lambda = \max_{i\neq j, \mia \asymp \mja} \limsup_{\alpha \to + \infty} \frac{\mia}{\mja} \]
where, for $i \neq j$, $ \mia \asymp \mja$ means that there exists $C >0$ independent of $\alpha$ such that $\frac{1}{C} \mia \le \mja \le C \mia$ up to a subsequence. Assume that there exist $i \neq j$ such that $ \mia \asymp \mja$. The structure relation \eqref{structure} shows that $\frac{d_g(\xia, \xja)}{\mia} \to + \infty$ and we thus have by \eqref{defsij} and \eqref{defri} that for $\alpha$ large enough
\[ r_{i, \alpha} + r_{j,\alpha} \le s_{i,j,\alpha} + s_{j,i,\alpha} \le 2 \sqrt{\frac{ \Lambda}{n(n-2)}} d_g(\xia, \xja). \] 
Let $R_0 >2$ be large enough so that, up to a subsequence, 
\ben \label{defR0}
\begin{aligned}
& \textrm{ for all } i \neq j \textrm{ such that }  \mia \asymp \mja, \quad \frac{2 \ria}{R_0} \le \frac14 d_g(\xia, \xja). 
\end{aligned}
 \een
In particular
\[ B_{g}\Big(\xia, \frac{2 r_{i,\alpha}}{R_0} \Big) \cap B_{g}\Big(\xja, \frac{2 r_{j,\alpha}}{R_0} \Big) = \emptyset \]
for all $i \neq j$ with $ \mia \asymp \mja$. 
The following lemma follows from the definitions of $\ria$ and $R_0$ and its proof is omitted:
 
 \begin{lemme} \label{compbulles}
 Let $i \in \{1, \dots, k\}$ and $j \in \mathcal{A}_i$. Then 
 \[ B_{j,\alpha}(x) \lesssim \frac{\mia^{\pui}}{s_{i,j,\alpha}^{n-2}} \quad \textrm{ and } \quad  |\nabla B_{j,\alpha}(x)| \lesssim \frac{\mia^{\pui}}{s_{i,j,\alpha}^{n-1}} \quad \textrm{ for all } x \in  B_{g}\Big(\xia, \frac{2 r_{i,\alpha}}{R_0} \Big),\]
 where $B_{j,\alpha}$ is as in \eqref{defBia}. In particular, 
 \[ |V_{j,\alpha}| \lesssim \Bia \quad \textrm{ and } \quad |\nabla V_{j,\alpha}|   \lesssim \tia^{-1} \Bia\]
  in $B_{g}\Big(\xia, \frac{2 r_{i,\alpha}}{R_0} \Big)$, where $V_{j,\alpha}$ is as in \eqref{Via} and $\tia$ as in \eqref{defti}.
 \end{lemme}
  
We still define here the radii of influence $\ria$ by comparing the heights of the positive bubbles $\Bia$, just as in the positive case of Druet \cite{DruetJDG} (see also Druet-Hebey \cite{DruetHebey2}). This is not surprising since, although $V_{i,\alpha}$ might vanish, its value at a generic point is of the order of $\Bia$ by  \eqref{contVa} and  \eqref{Via}. We will refer to the ball $B_g(\xia, \frac{2\ria}{R_0})$ as \emph{the area of influence of $V_{i,\alpha}$}. 

\medskip

Any bubble in the bubble-tree also interacts pointwise with bubbles concentrating strictly faster -- the so-called \emph{higher} bubbles. In our proof of Theorem \ref{proplin} only the higher bubbles interacting with $\Bia$ in its area of influence will have to be taken into account. Let $i \in \{1, \dots, k \}$. We define:
\ben \label{defBi}
\mathcal{B}_i = \Big \{ j \in \{1, \dots, k \} \backslash \{i\} \textrm{ s. t. } \mja = o(\mia) \textrm{ and } d_g(\xia, \xja) \le \frac{2 \ria}{R_0} \Big \}.
\een
We then have the following lemma: 
\begin{lemme}
For any $x \in  B_g(\xia, \frac{3 \ria}{2R_0})$ we have 
\ben \label{calcul3}
\begin{aligned}
|\mWa(x)| & \lesssim \sum_{j=0}^k \Bja(x) \lesssim \Bia(x) + \sum_{j \in \mathcal{B}_i} \Bja(x) \quad \textrm{ and } \\
|\nabla \mWa(x)| & \lesssim B_{0,\alpha}(x) +  \sum_{j=1}^k \tja(x)^{-1}\Bja(x)  \\
&\lesssim \tia(x)^{-1} \Bia(x) + \sum_{j \in \mathcal{B}_i} \tja(x)^{-1}\Bja(x) .
 \end{aligned}
 \een
 \end{lemme}
\begin{proof}
By \eqref{defWa} and \eqref{contVa}  we have, for any $x \in M$
\[ |\mWa| \lesssim  B_{0,\alpha}(x) + \sum_{j=1}^k \Bja(x) \textrm{ and } |\nabla \mWa| \lesssim  B_{0,\alpha}(x) + \sum_{j=1}^k \tja(x)^{-1}\Bja(x) .\]
By \eqref{defti} we have $\tia \lesssim 1$ in $M$ and by definition of $\ria$ we always have $\Bia(x) \gtrsim B_{0,\alpha}(x)$ for $x \in B_g(\xia, \frac{2 \ria}{R_0})$. Let $j \neq i$. If $j \in \mathcal{A}_i$, Lemma \ref{compbulles} shows that $\Bja \lesssim \Bia$ and $|\nabla \Bja| \lesssim \tia^{-1} \Bia$ in $B_g(\xia, \frac{2 \ria}{R_0})$. Assume now that $\mja = o(\mia)$ but $j \not \in \mathcal{B}_i$. Then $d_g(\xia, \xja) > \frac{2 \ria}{R_0}$, thus $d_g(\xia, \xja) >> \mja$ and then for all $y \in B_g(\xia, \frac{3 \ria}{2R_0})$
\[ \Bja(y) \lesssim \frac{ \mja^{\pui} }{\ria^{n-2}} \lesssim  \frac{ \mia^{\pui} }{\ria^{n-2}} \lesssim \Bia(y) \textrm{ and }  |\nabla \Bja(y)| \lesssim \frac{ \mja^{\pui} }{\ria^{n-1}} \lesssim  \tia(y)^{-1} \Bia(y).  \]
\end{proof}

Let $i \in \{1, \dots, k\}$ and $j \in \mathcal{B}_i$. In view of \eqref{calcul3} it is important to understand when the profile $\Bja$ becomes dominant with respect to $\Bia$. It is an easy consequence of \eqref{defsij} and $\mja = o(\mia)$ that there exists $C>1$, only depending on $n$ and $(M,g)$ such that for all $y \in M$, 
 \ben \label{calcul}
\begin{aligned}
& \textrm{ if } d_g(\xja, y) \le \sjia \quad \textrm{ then } \quad \Bja(y) \ge \frac{1}{C}\Bia(y), \quad  \textrm{ and } \\
&  \textrm{ if }  \Bja(y) \ge \frac{1}{C}\Bia(y) \quad \textrm{ then } \quad d_g(\xja,y) \le  Cs_{j,i,\alpha}.
\end{aligned}
 \een
With \eqref{calcul3} a consequence is that, for any $x \in  B_g(\xia, \frac{3 \ria}{2 R_0}) \backslash \bigcup_{j \in \mathcal{B}_i} B_g(\xja, \sjia)$,
\ben \label{calcul5}
|\mWa(x)| \lesssim \sum_{j=0}^k \Bja(x) \lesssim \Bia(x)
 \een
holds. We will need a similar result for the derivatives of $V_{i,\alpha}$. Let $i \in \{1, \dots, k\}$ and $j \in \mathcal{B}_i$. We define a new radius $\rho_{j,i,\alpha}$ as follows:
\ben \label{defrhoia}
\bal 
\rjia & = 2  \left( \frac{\mja}{\mia} \right)^{\frac{n-2}{2(n-1)}} \big( d_g(\xia, \xja) + \mia \big).
\eal 
\een
If $j \in \mathcal{B}_i$, direct computations show again that there exists $C = C(n,g) >0$ such that  
\ben \label{calcul2}
\begin{aligned}
& \textrm{ if }  d_g(\xja,y) \le \rjia 
 \textrm{ then }  \tja(y)^{-1}\Bja(y) \ge \frac{1}{C} \tia(y)^{-1} \Bia(y) , \quad \textrm{ and }\\ 
& \textrm{ if }  \tja(y)^{-1}\Bja(y) \ge \frac{1}{C} \tia(y)^{-1} \Bia(y) \textrm{ then } d_g(\xja, y) \le C \rjia. 
 \end{aligned}
 \een
Note that, by definition, $\rjia = o(\ria)$ for $j \in \mathcal{B}_i$. Direct computations with \eqref{defsij} show that $ \rjia= 2  \left( \frac{\mja}{\mia} \right)^{-\frac{1}{2(n-1)}}\sjia$.  Since $j \in \mathcal{B}_i$ we have $\mja = o(\mia)$ and hence $\sjia = o(\rjia)$. With \eqref{calcul3}, \eqref{calcul5} and \eqref{calcul2} we obtain
\ben \label{calcul4}
\begin{aligned}
 |\mWa(x)| + \tia(x) &|\nabla \mWa(x)| \lesssim \Bia(x) \\
 & \quad \textrm{ for all } x \in B_g(\xia, \frac{3\ria}{2 R_0}) \backslash \bigcup_{j \in \mathcal{B}_i} B_g(\xja, \rjia).
\end{aligned}
  \een
 If finally $i \in \{1, \dots, k \}$ we define $\mathcal{C}_i$ to be the subset of $\mathcal{B}_i$ given by
\ben \label{defCi}
\mathcal{C}_i = \big \{ j \in \mathcal{B}_i \textrm{ such that } d_g(\xia, \xja) = O(\mia) \big \}.
\een
When $j \in \mathcal{C}_i$  one has $\sjia = O(\sqrt{\mia \mja})$ and $\rjia = O(\mja^{\frac{n-2}{2(n-1)}} \mia^{\frac{n}{2(n-1)}})$.

\subsection{Proof of Theorem \ref{proplin} part $1$: an inductive estimate} \label{proof1}

Let $(h_\alpha)_\alpha$ be a sequence of functions converging towards $h$ in $L^\infty(M)$, $k \ge 1$, $V_1, \dots, V_k $ in $D^{1,2}(\R^n)$ be $k$ solutions of \eqref{yamabe}, $u_0$ be a solution of \eqref{limiteq} and $(\xia)_{\alpha}$ and $(\mia)_\alpha$, $1 \le i \le k$ be $k$ sequences of points in $M$ and of positive real numbers satisfying \eqref{structure}. Let $(\mWa)_\alpha$ be as in \eqref{defWa}, let $(\ea)_\alpha$, $(\tau_\alpha)_\alpha$, $(\gamma_\alpha)_\alpha$ and $(R_\alpha)_\alpha$ be as in the statement of Theorem \ref{proplin}, and let $\vpa$ be the unique solution of \eqref{eqvpa}. We keep in the whole proof of Theorem \ref{proplin} the notations of subsections $3.2$ and $3.3$ and we will repeatedly use \eqref{ineqq} and \eqref{defsigmaa}.

\medskip 

We first prove \eqref{vp} which is easily obtained. Direct computations using  \eqref{propRa} give 
\[ \Vert  R_\alpha  \Vert_{L^{\frac{2n}{n+2}}(M)} \lesssim  \tau_\alpha + \ea +\sa \gamma_\alpha \]
for some sequence $(\sa)_\alpha$ satisfying \eqref{defsigmaa}, so that by Proposition \ref{propH1} $\vpa$ satisfies 
\[ \bal 
\Vert \vpa \Vert_{H^1(M)} & \lesssim \Vert (\triangle_g + 1)^{-1}(  R_\alpha) \Vert_{H^1(M)} \\
& \lesssim \Vert R_\alpha \Vert_{L^{\frac{2n}{n+2}}(M)}   \lesssim \tau_\alpha + \ea +\sa \gamma_\alpha .
\eal \]
Integrating \eqref{eqvpa} against $Z_{i,j,\alpha}$ for all $0 \le i \le k$ and $1 \le j \le n_i$ then yields, using \eqref{orthoZij} and the latter inequalities, that
\ben \label{proplin2}
 \sum \limits_{\underset{j =1 .. n_i }{i=0..k}} |\lija|  \lesssim \tau_\alpha + \ea +\sa \gamma_\alpha ,
\een
which proves \eqref{vp}.

\medskip

For $\alpha \ge 1$ we now let:
\ben \label{defaa}
a_{\alpha} = \left | \left| \frac{\vpa}{\sum_{i=0}^k \Bia}\right | \right|_{L^\infty(M)} + \left| \left|  \frac{\nabla \vpa}{B_{0,\alpha} + \sum_{i=1}^k \tia^{-1} \Bia}\right| \right|_{L^\infty(M)}
\een
where the $\Bia$ are as in \eqref{defBia} and $\tia$ as in \eqref{defti}. The proof of Theorem \ref{proplin} goes through a top-down induction in the bubble-tree. Starting from the highest bubbles we improve the naive control given by \eqref{defaa} on the area of influence of each bubble. We show that, up to some boundary terms, $\vpa$ is negligible with respect to $\Bia$ on $B_g(\xia, \frac{\ria}{R_0})$. These improved estimates are then propagated to the lower bubbles, and the induction parameter is the height in the bubble-tree. In this way we obtain refined estimates on $\vpa$ over the whole $M$ where $a_\alpha$ only appears in negligible terms. These estimates self-improve and yield \eqref{estopt}. 

\medskip

Up to renumbering the bubbles $V_{1, \alpha}, \dots, V_{k,\alpha}$ and up to passing to a subsequence for $(\xia)_\alpha$ and  $(\mia)_\alpha$, $
1 \le i \le k$, we assume that
\[ \mu_{1, \alpha} \ge \dots \ge \mu_{k,\alpha} \]
and we separate the family of bubbles $V_{1,\alpha}, \dots, V_{k,\alpha}$ in subfamilies concentrating at comparable speeds. Precisely, we define inductively 
\[ E_1 = \big \{i \in \{1, \dots, k\} \textrm{ such that }  \mu_{i,\alpha} \asymp \mu_{1, \alpha}\big  \}  = \big \{1, \dots, N_1 \big \} \]
for some $1 \le N_1 \le k$ and, assuming that we have defined $N_q$ for $q \ge 1$, we let 
\[ E_{q+1} = \big \{ i \ge N_q +1 \textrm{ such that } \mia \asymp \mu_{N_q+1, \alpha} \big \} = \big \{N_q + 1, \dots, N_{q+1} \} \]
for some $N_q + 1 \le N_{q+1} \le k $. There exists an integer $1 \le p \le k$ such that 
\ben \label{defp}
 \{1, \dots, k \} = E_1 \cup \dots \cup E_p . 
 \een
In a given subset $E_j$ all the bubbles concentrate with comparable speeds. We say that bubbles in $E_j$ have height $j$ in the bubble-tree. Note that defining $E_1, \dots, E_p$ as in \eqref{defp} does not involve $\vpa$ or the sequences $(\ea)_\alpha$, $(\tau_\alpha)_\alpha$, $(\gamma_\alpha)_\alpha$ and $(R_\alpha)_\alpha$ but only the configuration $V_{1, \alpha}, \dots, V_{k,\alpha}$.

\medskip

Let $\ell \in \{1, \dots, p\}$, where $p$ is as in \eqref{defp}. We consider the following two statements:

\bigskip

\textbf{Property $(H_{\ell,1})$:} There exists a sequence $(\sa^\ell)_\alpha$ as in  \eqref{defsigmaa} such that for all $q \ge \ell$, for all $i \in E_q$ and for all  $x \in B_g(\xia, \frac{\ria}{R_0})$ we have 
\ben \label{Hl1} \tag{$H_{\ell,1}$}
\begin{aligned}
|\vpa(x)| &\lesssim \big(\tau_\alpha + \ea  +  \sa^\ell( \gamma_\alpha + a_\alpha) \big)\Big( \Bia(x) + \sum_{j \in \mathcal{B}_i} \Bja(x) \Big) + a_\alpha \frac{\mia^{\pui}}{\ria^{n-2}}\\
& +\gamma_\alpha  \sum_{i=1}^k \Tia(x) B_i(x)
\end{aligned} 
\een
where $\mathcal{B}_i$ is as in \eqref{defBi}, $\tia$ as in \eqref{defti} and $\Tia$ as in \eqref{defTi}.

\textbf{Property $(H_{\ell,2})$:} There exists a sequence $(\sa^\ell)_\alpha$ as in  \eqref{defsigmaa} such that for all $q \ge \ell$, for all $i \in E_q$ and for all $x \in M$ we have 
\ben \label{Hl2} \tag{$H_{\ell,2}$}
\begin{aligned}
&\int_{B_g(\xia, \frac{\ria}{R_0})} d_g(x, \cdot)^{2-n} \big( \sum_{i=0}^k \Bia \big)^{2^*-2}|\vpa| dv_g \\
 &\lesssim \big(\tau_\alpha + \ea  +  \sa^\ell( \gamma_\alpha + a_\alpha) \big)\Big( \Bia(x) + \sum_{j \in \mathcal{B}_i} \Bja(x) \Big)  \quad \textrm{ and } \\
 &\int_{B_g(\xia, \frac{\ria}{R_0})} d_g(x, \cdot)^{1-n} \big( \sum_{i=0}^k \Bia \big)^{2^*-2}|\vpa| dv_g \\
 &\lesssim \big(\tau_\alpha + \ea  +  \sa^\ell( \gamma_\alpha + a_\alpha) \big) \Big( B_{0,\alpha}(x) + \sum_{i=1}^k \tia(x)^{-1}\Bia(x) \Big).
 \end{aligned} 
\een
Note that if $q \in \{1, \dots, p \}$, $i \in E_q$ and $x \in B_g(\xia, \frac{\ria}{R_0})$, then by \eqref{calcul3} we have
\[  \Bia(x) + \sum_{j \in \mathcal{B}_i} \Bja(x) \lesssim \sum_{j \in E_s, s \ge q} \Bja(x) \lesssim \Bia(x) + \sum_{j \in \mathcal{B}_i} \Bja(x).  \]
The right-hand side of \eqref{Hl1} thus states that $\vpa$ is negligible with respect to the sum of all the (positive) bubbles in $\cup_{q \ge \ell} E_q$. The induction process is given by the following Proposition, that we prove in the next subsection:

\begin{prop} \label{recurrence}
For any $1 \le \ell \le p$, $(H_{\ell,1})$ and $(H_{\ell,2})$ are true.
\end{prop}

\subsection{Proof of Proposition \ref{recurrence}.} \label{proof2}

We prove Proposition \ref{recurrence} by reverse induction on $\ell$. The proofs of the initial case $\ell = p$ and of the induction step are very similar, hence we prove them together.  Let $\ell \in \{1, \dots, p\}$. If $\ell \le p-1$ we assume that $(H_{\ell+1,1})$ and $(H_{\ell+1,2})$ hold true, while if $\ell = p$ we do not assume anything. Let $i \in E_\ell$ where $E_\ell$ is as in  \eqref{defp}. 

\medskip

\textbf{Step $1$: improved estimates in $B_g(\xia, \frac{\ria}{R_0})$.}  Let $G_h$ be the Green's function of $\triangle_g + h$ in $M$. It is by definition orthogonal to the kernel of $\triangle_g +h$ and by standard results (see e.g. Aubin \cite{Aub} or Robert \cite{RobDirichlet}) there exists $C>0$ only depending on $(M,g)$ and $h$ such that $|G_h(x,y)| \le C d_g(x,y)^{2-n}$ and $|\nabla_{y} G_h(x,y)| \le C d_g(x,y)^{1-n}$ for all $x \neq y$ in $M$. First, using \eqref{greenB} and \eqref{greenBij} below and \eqref{propRa} it is easily seen that there exists a sequence $(\sa)_\alpha$ as in \eqref{defsigmaa} such that for any sequence $(\xa)_\alpha \in M$
\ben \label{proplin1}
\begin{aligned}
& \int_M G_h(\xa, \cdot) R_\alpha dv_g \\
& \lesssim \big(\tau_\alpha + \ea  +  \sa^\ell \gamma_\alpha \big)\sum_{i=0}^k \Bia(\xa)  +\gamma_\alpha \sum_{i=1}^k \Tia(\xa) B_i(\xa) 
\end{aligned}
\een
holds. We write a representation formula for $\vpa$ in the whole of $M$. Since $M$ has no boundary we have, for any sequence  $(\xa)_\alpha \in B_g(\xia, \frac{\ria}{R_0})$:
\[\begin{aligned}
 \vpa(\xa) - \Pi_{K_h}(\vpa)(\xa) & = \int_{M} G_h(\xa, \cdot) (\triangle_g + h) \vpa dv_g, \\ 
  \end{aligned} \]
 where we have let $K_h = \ker(\triangle_g + h)$ and where $\Pi_{K_h}$ denotes the orthogonal projection on $K_{h}$ for the $H^1(M)$ scalar product. We first claim that 
 \ben \label{proplin101}
 \big|\Pi_{K_h}(\vpa)(\xa)\big| \lesssim  \frac{\mia^{\pui}}{\ria^{n-2}} a_\alpha
 \een 
 holds. Indeed, integrating by parts and using \eqref{defaa} yields
 \[ \big| \Pi_{K_h}(\vpa)(\xa) \big| \lesssim \int_M |\vpa|dv_g \lesssim a_\alpha B_{0,\alpha}(\xa) + \sum_{j=1}^k \mja^{\pui} a_\alpha.   \]
If $B_{0,\alpha} \not \equiv 0$ we have by \eqref{defri} $\ria \le \sqrt{\mia}$ and hence $a_\alpha B_{0,\alpha}(\xa) \lesssim \frac{\mia^{\pui}}{\ria^{n-2}} a_\alpha$. Since $\ria \lesssim 1$ we also have $\mja^{\pui} a_\alpha \lesssim  \frac{\mia^{\pui}}{\ria^{n-2}} a_\alpha$ for all $j$ such that $\mja \lesssim \mia$. Finally, if $\mja >> \mia$, and since $\xa \in B_g(\xia, \frac{2 \ria}{R_0})$, we use Lemma \ref{compbulles} to write that 
\[ a_\alpha \mja^{\pui} \lesssim a_\alpha \Bja(\xa) \lesssim  \frac{\mia^{\pui}}{\ria^{n-2}} a_\alpha. \]
These considerations prove \eqref{proplin101}. We now claim that 
 \ben \label{proplin102}
\bal
 & \int_{M\backslash B_g(\xia, \frac{3\ria}{2R_0})} d_g(\xa,\cdot)^{2-n}  \big( \sum_{i=0}^k \Bia)^{2^*-2}|\vpa| dv_g \\
 &\lesssim  \sa a_\alpha \Big(  \Bia(\xa) + \sum_{j \in \mathcal{B}_i} \Bja(\xa) \Big) + \frac{\mia^{\pui}}{\ria^{n-2}} a_\alpha
\eal
 \een
for some sequence $(\sa)_\alpha$ satisfying \eqref{defsigmaa}. Let $\tilde{\mathcal{B}}_i \subset \mathcal{B}_i$ be defined by 
\[ \tilde{\mathcal{B}}_i = \Big \{ j \in \mathcal{B}_i \textrm{ such that } d_g(\xia, \xja) \le \frac{5\ria}{4R_0} \Big \}  .\]
Let also 
\[ \mathcal{D}_i = \Big \{ j \in \{1, \dots, k\} \textrm{ such that } \mja = o(\mia) \textrm{ and }  d_g(\xia, \xja) > \frac{5\ria}{4R_0} \Big \}.  \]
We can write:
\ben \label{proplin103}
 \bal
 & \int_{M\backslash B_g(\xia, \frac{3\ria}{2R_0})} d_g(\xa,\cdot)^{2-n}  \big( \sum_{i=0}^k \Bia)^{2^*-2}|\vpa| dv_g \\
 & \lesssim \int_{M\backslash B_g(\xia, \frac{3\ria}{2R_0})}  a_\alpha d_g(\xa,\cdot)^{2-n} \times \\
 &\quad \quad \Big( \Bia^{2^*-1} + \sum_{j \in \tilde{\mathcal{B}}_i} \Bja^{2^*-1}  + \sum_{j \in \mathcal{D}_i} \Bja^{2^*-1} + \sum_{j \in \mathcal{A}_i} \Bja^{2^*-1} \Big) dv_g \\
 &\lesssim \sa a_\alpha \Big( \Bia(\xa) + \sum_{j \in \mathcal{B}_i} \Bja(\xa) \Big)+ a_\alpha \sum_{j \in \mathcal{D}_i} \Bja(\xa)  \\
 & + a_\alpha \sum_{j \in \mathcal{A}_i} \Bja(\xa),
 \eal 
\een
 where the last line follows from \eqref{decbulle} below. With Lemma \ref{compbulles}, and since $\xa \in B_g(\xia, \frac{\ria}{R_0})$, we have $ a_\alpha \sum_{j \in \mathcal{A}_i} \Bja(\xa) \lesssim  \frac{\mia^{\pui}}{\ria^{n-2}} a_\alpha$. If $j \in \mathcal{D}_i$ and since $\xa \in B_g(\xia, \frac{\ria}{R_0})$, we have $d_g(\xja, \xa) \gtrsim \ria$, so that
 \[  a_\alpha \sum_{j \in \mathcal{D}_i}  \Bja(\xa) \lesssim \frac{\mia^{\pui}}{\ria^{n-2}} a_\alpha.\]
 Together with \eqref{proplin103} this proves \eqref{proplin102}. Using \eqref{proplin2}, \eqref{proplin1}, \eqref{proplin101}, \eqref{proplin102} and \eqref{greenB} below the representation formula gives, for any sequence $(\xa)_\alpha$ in $B_g(\xia, \frac{\ria}{R_0})$, 
\ben \label{proplin3}
\begin{aligned}
& |\vpa(\xa)| \\
& \lesssim \big(\tau_\alpha + \ea  +  \sa^\ell( \gamma_\alpha + a_\alpha) \big)\Big( \Bia(\xa) + \sum_{j \in \mathcal{B}_i} \Bja(\xa) \Big) \\
&+ a_\alpha \frac{\mia^{\pui}}{\ria^{n-2}}   +\gamma_\alpha \sum_{j=1}^k \Tja(\xa) \Bja(\xa)  \\
&+ \int_{B_g(\xia, \frac{3 \ria}{2R_0})} d_g(\xa, \cdot)^{2-n} |\mathcal{W}_\alpha|^{2^*-2}|\vpa| dv_g. 
\end{aligned}
\een
Note that the sum $\sum_{j \in \mathcal{B}_i} \Bja(x)$ vanishes when $\ell = p$.  To estimate the integral in \eqref{proplin3} we remark that 
\[ \begin{aligned}
 &B_g(\xia, \frac{3 \ria}{2 R_0})  \subseteq\\
& \Bigg[ \Big(B_g(\xia, \frac{3 \ria}{2 R_0}) \backslash B_g(\xia, \frac{ \ria}{R_0}) \Big) \backslash \bigcup_{j \in \mathcal{B}_i}   B_g(\xja, \rjia)  \Bigg] \\
 &  \bigcup \Bigg[  B_g(\xia, \frac{ \ria}{R_0})  \backslash \bigcup_{j \in \mathcal{B}_i}   B_g(\xja, \rjia) \Bigg] \\
  & \bigcup_{j \in \mathcal{B}_i}  \Big( B_g(\xja, \rjia) \backslash \cup_{s \in \mathcal{B}_i} B_g(x_{s, \alpha}, \frac{r_{s,\alpha}}{R_0}) \Big) \cap B_g(\xia, \frac{3 \ria}{2 R_0}) \Big) \\
  & \bigcup_{j \in \mathcal{B}_i } B_g(x_{j, \alpha}, \frac{r_{j,\alpha}}{R_0}) 
  \end{aligned}
 \]
 where $\rjia$ is as in \eqref{defrhoia}. Let $y \in \Big(B_g(\xia, \frac{3 \ria}{2 R_0}) \backslash B_g(\xia, \frac{ \ria}{R_0}) \Big) \backslash \bigcup_{j \in \mathcal{B}_i}   B_g(\xja, \rjia) $. By \eqref{calcul4} we have 
 \[ |\mathcal{W}_\alpha(y)| \lesssim \Bia(y) \lesssim \frac{\mia^{\pui}}{\ria^{n-2}}. \]
Straightforward computations with \eqref{defaa} then yield
 \ben \label{proplin4}
\begin{aligned}
 \int \limits_{\big(B_g(\xia, \frac{3 \ria}{2 R_0}) \backslash B_g(\xia, \frac{ \ria}{R_0}) \big) \backslash \bigcup_{j \in \mathcal{B}_i}   B_g(\xja, \rjia) }&d_g(\xa, \cdot)^{2-n} |\mathcal{W}_\alpha|^{2^*-2}|\vpa| dv_g  \\
 & \lesssim   \sa a_\alpha \Bia(\xa).
\end{aligned}
 \een
By definition of $\mathcal{W}_\alpha$ we have $|\mathcal{W}_\alpha| \lesssim \sum_{i=0}^k \Bia$. Since $i \in E_\ell$, if $j \in \mathcal{B}_i$ we have $j \in E_{q}$ for some $q \ge \ell+1$, where $E_{\ell}$ and $E_q$ are as in \eqref{defp}. The induction property $(H_{\ell+1, 2})$ together with \eqref{calcul4} then shows that there exists a sequence $(\sa^{\ell+1})_\alpha$ satisfying \eqref{defsigmaa} such that 
 \ben \label{proplin5}
 \begin{aligned}
& \int \limits_{\bigcup_{j \in \mathcal{B}_i } B_g(x_{j, \alpha}, \frac{r_{j,\alpha}}{R_0}) } d_g(\xa, \cdot)^{2-n} |\mathcal{W}_\alpha|^{2^*-2}|\vpa| dv_g \\
 &\lesssim \big(\tau_\alpha + \ea  +  \sa^{\ell+1}( \gamma_\alpha + a_\alpha) \big)\Big( \Bia(\xa) + \sum_{j \in \mathcal{B}_i} \Bja(\xa) \Big). \\
 \end{aligned}
 \een
The latter integral only has to be computed $\ell \le p-1$, otherwise it vanishes. Remark now that, by \eqref{calcul3} and \eqref{defaa}, we have
\be
 |\mathcal{W}_\alpha|^{2^*-2} |\vpa| \lesssim a_\alpha \Bia^{2^*-1} + a_\alpha  \sum_{j \in \mathcal{B}_i } \Bja^{2^*-1}
 \ee
in  $ \bigcup_{j \in \mathcal{B}_i}  \Big( B_g(\xja, \rjia) \backslash \cup_{s \in \mathcal{B}_i} B_g(x_{s, \alpha}, \frac{r_{s,\alpha}}{R_0}) \Big) \cap B_g(\xia, \frac{3\ria}{2 R_0}) $. On the one hand, straightforward computations using \eqref{decbulle} below give 
\[\begin{aligned}
&\int_{ \bigcup_{j \in \mathcal{B}_i}  \Big( B_g(\xja, \rjia) \backslash \cup_{s \in \mathcal{B}_i} B_g(x_{s, \alpha}, \frac{r_{s,\alpha}}{R_0}) \Big)\cap B_g(\xia, \frac{3 \ria}{2 R_0})  } d_g(\xa, \cdot)^{2-n} a_\alpha  \sum_{j \in \mathcal{B}_i } \Bja^{2^*-1} dv_g \\
& \lesssim 
\sa a_\alpha  \sum_{j \in \mathcal{B}_i} \Bja(\xa) 
\end{aligned}
 \]
 for some sequence $(\sa)_\alpha$ satisfying \eqref{defsigmaa}. On the other hand, if $j \in \mathcal{B}_i$, 
  \[ \begin{aligned}
&  \int_{  B_g(\xja, \rjia) \backslash \cup_{s \in \mathcal{B}_i} B_g(x_{s, \alpha}, \frac{r_{s,\alpha}}{R_0})\cap B_g(\xia, \frac{ 3\ria}{2 R_0})  } d_g(x, \cdot)^{2-n}  a_\alpha \Bia^{2^*-1} dv_g\\ 
& \le  \int_{   B_g(\xja, \rjia) } d_g(x, \cdot)^{2-n}  a_\alpha \Bia^{2^*-1} dv_g \\
& \lesssim \left \{ \begin{aligned} & \left( \frac{\rjia}{\mia}\right)^2 a_\alpha \Bia(x) & \textrm{ if } d_g(\xia, \xja) = O(\mia) \\  & \left( \frac{\mia}{d_g(\xia, \xja)}\right)^2 a_\alpha \Bia(x) & \textrm{ otherwise }  \end{aligned} \right. \\
& \lesssim \sa a_\alpha \Bia(x).
 \end{aligned}
 \]
We have used that, by \eqref{defrhoia}, we have $\rjia = o(\mia)$ in the first case and $\rjia = o(d_g(\xia, \xja))$ in the second case. Combining the last two estimates shows that
\ben \label{proplin6}
\begin{aligned}
&\int_{ \bigcup_{j \in \mathcal{B}_i}  \Big( B_g(\xja, \rjia) \backslash \cup_{s \in \mathcal{B}_i} B_g(x_{s, \alpha}, \frac{r_{s,\alpha}}{R_0}) \Big) \cap B_g(\xia, \frac{ 3\ria}{2 R_0}) } d_g(\xa, \cdot)^{2-n} |\mathcal{W}_\alpha|^{2^*-2}|\vpa| dv_g \\
& \lesssim  \sa a_\alpha\Big( \Bia(\xa) + \sum_{j \in \mathcal{B}_i} \Bja(\xa) \Big) 
\end{aligned}
\een
for some sequence $(\sa)_\alpha$ satisfying \eqref{defsigmaa}. Estimate \eqref{proplin6} shows that in the representation formula the ``neck regions'' $ B_g(\xja, \rjia) \backslash \cup_{s \in \mathcal{B}_i} B_g(x_{s, \alpha}, \frac{r_{s,\alpha}}{R_0})$ for $j \in \mathcal{B}_i$ only contribute to  $\vpa$ with a term that is negligible compared to  $\Bia + \sum_{j \in \mathcal{B}_i} \Bja$. This shows why we only need to prove the induction property on $B_g(\xia, \frac{\ria}{R_0})$ in \eqref{Hl1}. We finally let 
\ben \label{defOia}
\Omega_{i,\alpha} = B_g(\xia, \frac{ \ria}{R_0})  \backslash \bigcup_{j \in \mathcal{B}_i}   B_g(\xja, \rjia).
\een
With \eqref{calcul4} and Giraud's lemma we obtain 
\ben \label{proplin7} 
\int_{\Oia} d_g(\xa, \cdot)^{2-n} |\mathcal{W}_\alpha|^{2^*-2}|\vpa| dv_g \lesssim \left( \frac{\mia}{\tia(\xa)}\right)^2 \Vert \vpa \Vert_{L^\infty(\Oia)}. 
\een
Gathering \eqref{proplin4}, \eqref{proplin5}, \eqref{proplin6} and \eqref{proplin7} in \eqref{proplin3} we have thus shown that there exists a sequence $(\sa)_\alpha$ as in \eqref{defsigmaa} such that, for any sequence $(\xa)_\alpha \in B_g(\xia, \frac{\ria}{R_0})$,
\ben \label{proplin71} \begin{aligned}
& |\vpa(\xa)|  \lesssim \big(\tau_\alpha + \ea  +  \sa^\ell( \gamma_\alpha + a_\alpha) \big)\Big( \Bia(\xa) + \sum_{j \in \mathcal{B}_i} \Bja(\xa) \Big) \\ &+ a_\alpha \frac{\mia^{\pui}}{\ria^{n-2}}   +\gamma_\alpha \sum_{j=1}^k \Tja(\xa) \Bja(\xa)   + \left( \frac{\mia}{\tia(\xa)}\right)^2 \Vert \vpa \Vert_{L^\infty(\Oia)}\end{aligned}
\een
and 
\ben \label{proplin72} 
\begin{aligned}
 &\int_{B_g(\xia, \frac{3\ria}{2 R_0})} d_g(\xa, \cdot)^{2-n}\Big( \sum_{i=0}^k \Bia \Big)^{2^*-2}|\vpa| dv_g \\
 &\lesssim  \left( \frac{\mia}{\tia(\xa)}\right)^2 \Vert \vpa \Vert_{L^\infty(\Oia)} \\
 &+ \big(\tau_\alpha + \ea  +  \sa^\ell( \gamma_\alpha + a_\alpha) \big)\Big( \Bia(\xa) + \sum_{j \in \mathcal{B}_i} \Bja(\xa) \Big)  
 \end{aligned}
\een
hold, where $\Oia$ is as in \eqref{defOia}. We can now use \eqref{proplin71} to compute again the right-hand side of \eqref{proplin7} with a greater precision. After a finite number of iterations we obtain that there exists a sequence $(\sa^{\ell,1})_\alpha$ as in \eqref{defsigmaa} such that
\ben \label{proplin8} 
\begin{aligned}
& |\vpa(\xa)|  \lesssim  a_\alpha \frac{\mia^{\pui}}{\ria^{n-2}} +\gamma_\alpha  \sum_{j=1}^k \Tja(\xa) \Bja(\xa) \\
& + \Big(\tau_\alpha + \ea  +  \sa^{\ell,1}( \gamma_\alpha + a_\alpha)  + \mia^{\pui}  \Vert \vpa \Vert_{L^\infty(\Oia)}\Big)\Big( \Bia(\xa) + \sum_{j \in \mathcal{B}_i} \Bja(\xa) \Big), \\  
 \end{aligned}
\een
and
\ben \label{proplin9}
\begin{aligned}
& \int_{B_g(\xia, \frac{3\ria}{2 R_0})} d_g(\xa, \cdot)^{2-n}\Big( \sum_{i=0}^k \Bia \Big)^{2^*-2}|\vpa| dv_g \\
& \lesssim \Big(\tau_\alpha + \ea  +  \sa^{\ell,1}( \gamma_\alpha + a_\alpha)  + \mia^{\pui}  \Vert \vpa \Vert_{L^\infty(\Oia)}\Big)\Big( \Bia(\xa) + \sum_{j \in \mathcal{B}_i} \Bja(\xa) \Big)
\end{aligned}
\een
hold for any sequence $(\xa)_\alpha \in B_g(\xia, \frac{\ria}{R_0})$.

 \medskip

 \textbf{Step $2$: Gradient estimates on $\Oia$.} First, using \eqref{propRa}, \eqref{greenBder} and \eqref{greenBijder} below we have, for any sequence $(\xa)_\alpha$ in $M$,
 \[
 \begin{aligned}
\int_M d_g(\xa, \cdot)^{1-n} R_\alpha dv_g &\lesssim  \big(\tau_\alpha + \ea  +  \sa \gamma_\alpha \big) \Big( B_{0,\alpha}(\xa) + \sum_{j=1}^k \tja(\xa)^{-1}\Bja(\xa) \Big)\\
& +\gamma_\alpha \sum_{j=1}^k \Xi_{j,\alpha}(\xa) \Bja(\xa) ,
\end{aligned}
\]
where $\Xia$ is as in \eqref{defXia} below. As in Step $1$ we write a representation formula for $\vpa$ in $M$, that we differentiate once with respect to the variable $x$. The details are very similar to those of Step $1$ and we omit most of them. We then obtain that, for any sequence $(\xa)_\alpha \in B_g(\xia, \frac{\ria}{R_0})$:
\ben \label{proplin81}
\begin{aligned}
& |\nabla \vpa(\xa)| \\
& \lesssim \big( \tau_\alpha + \ea  +  \sa^\ell( \gamma_\alpha + a_\alpha)  \big) \Big( B_{0,\alpha}(\xa) + \sum_{j=1}^k \tja(\xa)^{-1}\Bja(\xa) \Big) \\
& +\gamma_\alpha \sum_{j=1}^k \Xi_{j,\alpha}(\xa) \Bja(\xa) + a_\alpha \frac{\mia^{\pui}}{\ria^{n-1}}\\
& + \int_{B_g(\xia, \frac{3\ria}{2R_0})} d_g(\xa, \cdot)^{1- n} \Big( \sum_{i=0}^k \Bia \Big)^{2^*-2}|\vpa| dv_g. 
\end{aligned}
\een
The integral over $B_g(\xia, \frac{3 \ria}{2R_0})$ is computed by following the exact arguments that led to \eqref{proplin4},  \eqref{proplin5},  \eqref{proplin6} and  \eqref{proplin7}, where the induction argument now relies on the second inequality in \eqref{Hl2}. This gives in the end
\[
 \begin{aligned}
|\nabla \vpa(\xa)| & \lesssim\big( \tau_\alpha + \ea  +  \sa( \gamma_\alpha + a_\alpha)  \big) \Big(B_{0,\alpha}(\xa) +  \sum_{j=1}^k \tja(\xa)^{-1}\Bja(\xa) \Big)\\
&+ a_\alpha \frac{\mia^{\pui}}{\ria^{n-1}}   +\gamma_\alpha  \sum_{j=1}^k \Xi_{j,\alpha}(\xa) \Bja(\xa)   + \frac{\mia^2}{\tia(\xa)^3} \Vert \vpa \Vert_{L^\infty(\Oia)}\end{aligned}
\]
and 
\[ 
\begin{aligned}
 &\int_{B_g(\xia, \frac{3 \ria}{2R_0})} d_g(\xa, \cdot)^{1-n} \Big( \sum_{i=0}^k \Bia \Big)^{2^*-2}|\vpa| dv_g \lesssim  \frac{\mia^2}{\tia(\xa)^3} \Vert \vpa \Vert_{L^\infty(\Oia)} \\
 & +\Big( \tau_\alpha + \ea  +  \sa( \gamma_\alpha + a_\alpha)  + \mia^{\pui}  \Vert \vpa \Vert_{L^\infty(\Oia)}\Big)\Big(B_{0,\alpha}(\xa) +  \sum_{j=1}^k \tja(\xa)^{-1}\Bja(\xa) \Big)
\end{aligned}
\]
for any sequence $(\xa)_\alpha \in B_g(\xia, \frac{ \ria}{R_0})$, where $\Oia$ is as in \eqref{defOia} and $(\sa)_\alpha$ is a sequence satisfying \eqref{defsigmaa}. As before these two inequalities self-improve after a finite number of iterations. We thus obtain that there exists a sequence $(\sa^{\ell,2})_\alpha$ satisfying \eqref{defsigmaa} such that, for any sequence $(\xa)_\alpha \in B_g(\xia, \frac{\ria}{R_0})$, 
\ben \label{proplin821} 
\begin{aligned}
& |\nabla \vpa(\xa)|  \lesssim  a_\alpha \frac{\mia^{\pui}}{\ria^{n-1}} +\gamma_\alpha \sum_{j=1}^k \Xi_{j,\alpha}(\xa) \Bja(\xa) \\
&+  \Big( \tau_\alpha + \ea  +  \sa^{\ell,2}( \gamma_\alpha + a_\alpha)  + \mia^{\pui}  \Vert \vpa \Vert_{L^\infty(\Oia)}\Big) \times \\
& \quad \quad \Big(B_{0,\alpha}(\xa) +  \sum_{j=1}^k \tja(\xa)^{-1}\Bja(\xa) \Big) \\
 \end{aligned}
\een
and
\ben \label{proplin822}
\begin{aligned}
& \int_{B_g(\xia, \frac{3 \ria}{2 R_0})} d_g(\xa, \cdot)^{1-n}\Big( \sum_{i=0}^k \Bia \Big)^{2^*-2}|\vpa| dv_g \\
& \lesssim   \Big( \tau_\alpha + \ea  +  \sa^{\ell,2}( \gamma_\alpha + a_\alpha)  + \mia^{\pui}  \Vert \vpa \Vert_{L^\infty(\Oia)}\Big) \times \\
& \quad \quad \Big(B_{0,\alpha}(\xa) +  \sum_{j=1}^k \tja(\xa)^{-1}\Bja(\xa) \Big)
\end{aligned}
\een
hold.

 \medskip

 \textbf{Step $3$: almost-orthogonality.} By construction $\vpa \in K_\alpha^{\perp}$ and so, for any $1 \le m \le n_i$,
 \ben \label{proplin91}
 \int_M \langle \nabla \vpa, \nabla Z_{i,m,\alpha} \rangle + \vpa Z_{i,m,\alpha} dv_g = 0.
 \een
  Straightforward computations using \eqref{defaa} yield 
 \ben \label{proplin10}
 \left| \int_M  \vpa Z_{i,m,\alpha} dv_g = 0 \right| \lesssim \sa a_\alpha
 \een 
 for some sequence $(\sa)_\alpha$ satisfying \eqref{defsigmaa}. By definition of $Z_{i,m,\alpha}$ and by \eqref{estZ} direct computations give, with \eqref{defaa} and \eqref{gradij} below
 \ben \label{proplin11}
 \int_{M \backslash B_g(\xia, \frac{\ria}{R_0})} | \langle \nabla \vpa, \nabla Z_{i,m,\alpha} \rangle| dv_g \lesssim \sa a_\alpha.
 \een
 Let now $j \in \mathcal{B}_i$. By  \eqref{calcul2} and \eqref{defaa} we have, for any $y\in  B_g(\xja, \rjia)\cap B_g(\xia, \frac{3\ria}{2R_0})$, 
\[ |\nabla \vpa(y) | \lesssim \sum_{j \in \mathcal{B}_i} \tja(y)^{-1}\Bja(y). \]
Independently, by \eqref{estZ} one always has $|\nabla Z_{i,m,\alpha}|\lesssim \mia^{1- \frac{n}{2}}$ in $M$. Therefore
 \ben \label{proplin12}
 \int_{ \Big(\bigcup_{j \in \mathcal{B}_i}  B_g(\xja, \rjia) \Big)\cap B_g(\xia, \frac{3\ria}{2R_0})} 
 | \langle \nabla \vpa, \nabla Z_{i,m,\alpha} \rangle| dv_g \lesssim \sa a_\alpha.
 \een
Define now 
 \ben \label{defhOia}
  \hat{\Omega}_{i,\alpha} = \frac{1}{\mia} \exp_{\xia}^{-1}(\Oia) 
  \een
where $\Oia$ is as in \eqref{defOia} and, for $y \in \hat{\Omega}_{i,\alpha}$,
  \ben \label{deftvpa}
  \hvpa(y) = \vpa \big( \exp_{\xia}(\mia y) \big),
  \een
 where $\exp_{\xia}$ is the exponential map for $g$ centered at $\xia$. For $y \in \hat{\Omega}_{i,\alpha}$ one has 
 \[ |\exp_{\xia}^*g(\mia \cdot) - \xi |(y) \lesssim \mia^2 |y|^2 \]
 where $\xi$ is the euclidean metric. Direct computations with \eqref{defaa} show that 
\[ \begin{aligned}
 & \int_{\Omega_{i,\alpha} } \langle \nabla \vpa, \nabla Z_{i,m,\alpha} \rangle dv_g  =\mia^{\pui} \int_{ \hat{\Omega}_{i,\alpha}}  \langle \nabla \hvpa, \nabla Z_{i,m} \rangle_{\xi} dx + O(\sa a_\alpha).
  \end{aligned}
  \]
Combining the latter with \eqref{proplin91}, \eqref{proplin10}, \eqref{proplin11} and \eqref{proplin12} finally shows that there exists a sequence $(\sa^{\ell,3})_\alpha$ satisfying \eqref{defsigmaa} such that for all $\alpha$,
   \ben \label{proplin13}
\left| \int_{ \hat{\Omega}_{i,\alpha}}  \langle \nabla \hvpa, \nabla Z_{i,m} \rangle_{\xi} dx \right| \lesssim  \mia^{1 - \frac{n}{2}} \sa^{\ell,3} a_\alpha
 \een
  for all $1 \le m \le n_i$.

 \medskip
 
 \textbf{Step $4$: equation satisfied by $\hvpa$.}  Let $j \in \mathcal{B}_i$ and  let  
  \[ \hat{x}_{j,\alpha} = \frac{1}{\mia} \exp_{\xia}^{-1}(\xja).  \]
 If $j \in \mathcal{C}_i$, where $\mathcal{C}_i$ is as in \eqref{defCi}, we have $d_g(\xja, \xia) = O(\mia)$: hence $|\hat{x}_{i,\alpha}|\lesssim1$ and $\rjia =o(\mia)$ as $\alpha \to + \infty$ by \eqref{defrhoia}. We let in this case $\hat{x}_{j,\infty} = \lim_{\alpha \to + \infty} \hat{x}_{j,\alpha}$. Otherwise, again by \eqref{defrhoia}, we have $|\hat{x}_{j,\alpha}| \to + \infty$ and $\rjia = o( d_g(\xia, \xja))$. As $\alpha \to + \infty$, and up to passing to a  subsequence, the sets $\hat{\Omega}_{i,\alpha}$ thus converge to $\R^n \backslash \{(\hat{x}_{j,\infty})_{j\in \mathcal{C}_i}\}$. We prove the following lemma:
  
  \begin{lemme} \label{lemmeeq}
There exists a sequence $(\sa^{\ell,4})_\alpha$ satisfying \eqref{defsigmaa} such that 
 \ben \label{proplin14}
 \big | \big|  \triangle_{g_{i,\alpha}} \hvpa \big| \big|_{L^\infty(\hat{\Omega}_{i,\alpha})} \lesssim \big( \tau_\alpha + \ea  +  \sa^{\ell,4}( \gamma_\alpha + a_\alpha) \big)\mia^{1 - \frac{n}{2}} + \Vert \vpa \Vert_{L^\infty(\Oia)}
 \een
holds, where $\Oia$ and $\hat{\Omega}_{i,\alpha}$ are as in \eqref{defOia} and \eqref{defhOia}.
  \end{lemme}
In the lemma we have let $g_{i,\alpha} = \exp_{\xia}^*g(\mia \cdot)$, so that $g_{i,\alpha} \to \xi$ in $C^2_{loc}(\R^n)$.
 
 \begin{proof}[Proof of Lemma \ref{lemmeeq}]
By definition of $\hvpa$ one has, for all $y \in \hat{\Omega}_{i,\alpha}$,
\[ \triangle_{g_{i,\alpha}} \hvpa = \mia^2 \triangle_g \vpa \big( \exp_{\xia}(\mia y) \big). \]
The Laplacian of $\vpa$ is estimated with \eqref{eqvpa}. First, by \eqref{defOia} and by \eqref{calcul} we have
\ben \label{proplin141}
 \Bja(\exp_{\xia}(\mia y)) \lesssim \mia^{1-\frac{n}{2}} (1+|y|)^{2-n}
 \een
for all $j \neq i$, $j \ge 1$, and  for all $y \in \hat{\Omega}_{i,\alpha}$, so that by \eqref{calcul5} and \eqref{defaa} we have
\ben \label{proplin1411}
\big| \mia^2 h_\alpha \big( \exp_{\xia}(\mia y) \big) \hvpa(y) \big| \lesssim \sa \mia^{1 - \frac{n}{2}}
\een
for some sequence $(\sa)_\alpha$ satisfying \eqref{defsigmaa} and for all $y \in \hat{\Omega}_{i,\alpha}$. If $j \ge 1$, straightforward computations show that for all $1 \le m \le n_j$
\[ |(\triangle_g +1)Z_{j,m,\alpha}| \lesssim \Bja^{2^*-1} + \Bja \]
holds, while if $j=0$ we have, for all $1 \le m \le n_0$, $|(\triangle_g +1)Z_{0,m,\alpha}| \lesssim 1$. As a consequence, Lemma \ref{compbulles} and \eqref{proplin2} show that:
 \ben \label{proplin15}
 \begin{aligned} 
 \Big| \mia^2  \sum \limits_{\underset{m =1 .. n_i }{j=0..k}} \lambda_{j,m}^{\alpha} &( \triangle_g   + 1)Z_{j,m, \alpha}  \big( \exp_{\xia}(\mia y)\big) \Big| \\
 & \lesssim \big( \tau_\alpha + \ea  +  \sa \gamma_\alpha \big) \mia^{1 - \frac{n}{2}}
 \end{aligned}
 \een
 for some sequence $(\sa)_\alpha$ as in \eqref{defsigmaa}  and for all $y \in \hat{\Omega}_{i,\alpha}$. Using \eqref{propRa}, Lemma \ref{compbulles}, \eqref{calcul} and \eqref{proplin141} we get that 
 \[ \begin{aligned}
  \mia^2 \big| R_\alpha &\big( \exp_{\xia}(\mia y)\big) \big|  \lesssim \big(\tau_\alpha + \ea  +  \sa \gamma_\alpha \big)  \mia^{1 - \frac{n}{2}}  \\
  &+ \gamma_\alpha \mia^2  \sum \limits_{\underset{j \neq i}{j=0..k}} \big( B_{j, \alpha}^{2^*-2} \Bia + \Bia^{2^*-2} \Bja \Big)\big( \exp_{\xia}(\mia y)\big)
  \end{aligned}\]
holds for all $y \in \hat{\Omega}_{i,\alpha}$. Let now $j \in \{1, \dots, k\}$ with $j \neq i$. Assume first that $\mia = O(\mja)$, hence $j \in \mathcal{A}_i$. Lemma \ref{compbulles} then shows that 
\[\begin{aligned}
 \mia^2 \big( &B_{j, \alpha}^{2^*-2} \Bia + \Bia^{2^*-2} \Bja \Big)\big( \exp_{\xia}(\mia y)\big) \\
 & \lesssim \left(\frac{\mia}{\sija}\right)^4 \mia^{1 - \frac{n}{2}} (1+|y|)^{2-n} +  \frac{\mia^\pui}{\sija^{n-2}} (1+|y|)^{-4} \\
 & \lesssim \sa \mia^{1 - \frac{n}{2}}
\end{aligned} \]
for all $y \in \hat{\Omega}_{i,\alpha}$, where we used that $\frac{\mia}{\sija} \to 0$ as $\alpha \to + \infty$. Assume now that $\mja = o(\mia)$. Then we either have $d_g(\xia, \xja) > \frac{2 \ria}{R_0}$ or $j \in \mathcal{B}_i$. In the first case, $\Bja$ satisfies
\[ \Bja(x) \lesssim \frac{\mja^{\pui}}{\ria^{n-2}} \quad \textrm{ on } \Oia, \]
so that 
\[ \begin{aligned}
 \mia^2 \big( &B_{j, \alpha}^{2^*-2} \Bia + \Bia^{2^*-2} \Bja \Big)\big( \exp_{\xia}(\mia y)\big) \\
 & \lesssim \left( \frac{\mja}{\mia} \right)^{2} \left(\frac{\mia}{\ria}\right)^4 \mia^{1 - \frac{n}{2}} (1+|y|)^{2-n} +  \left( \frac{\mja}{\mia} \right)^{\pui} \frac{\mia^\pui}{\ria^{n-2}} (1+|y|)^{-4} \\
 & \lesssim \sa \mia^{1 - \frac{n}{2}}.
\end{aligned} \]
Assume finally that $j \in \mathcal{B}_i$. If $y \in \hat{\Omega}_{i,\alpha} \cap \{y \in \hat{\Omega}_{i,\alpha}, |y-\hat{x}_{j,\alpha}| > \frac{1}{4} \}$ then
\[ \Bja\big( \exp_{\xia}(\mia y) \big) \lesssim  \left( \frac{\mja}{\mia} \right)^{\pui} \mia^{1- \frac{n}{2}}, \]
so that 
\[ \begin{aligned}
 \mia^2 \big( &B_{j, \alpha}^{2^*-2} \Bia + \Bia^{2^*-2} \Bja \Big)\big( \exp_{\xia}(\mia y)\big) \\
 & \lesssim \left( \frac{\mja}{\mia} \right)^{2} \mia^{1 - \frac{n}{2}} (1+|y|)^{2-n} + \left( \frac{\mja}{\mia} \right)^{\pui} \mia^{1 - \frac{n}{2}} (1+|y|)^{-4} \\
 & \lesssim \sa \mia^{1 - \frac{n}{2}}.
\end{aligned} \]
If $y \in B(\hat{x}_{j,\alpha}, \frac14) \backslash B(\hat{x}_{j,\alpha}, \frac12 \frac{\rjia}{\mia})$ we then have 
\[\Bia\big( \exp_{\xia}(\mia y)\big) \gtrsim \frac{\mia^\pui}{(\mia + d_g(\xia, \xja))^{n-2}}  \]
and thus by \eqref{defrhoia}
\[ \frac{\Bja}{\Bia}\big( \exp_{\xia}(\mia y)\big) \lesssim \left( \frac{\mja}{\mia} \right)^{\frac{1}{n-1}}. \]
As a consequence
\[ \begin{aligned}
 \mia^2 \big( &B_{j, \alpha}^{2^*-2} \Bia + \Bia^{2^*-2} \Bja \Big)\big( \exp_{\xia}(\mia y)\big)\\
 & \lesssim  \left( \left( \frac{\mja}{\mia} \right)^{\frac{1}{n-1}} +  \left( \frac{\mja}{\mia} \right)^{\frac{4}{(n-2)(n-1)}} \right) \mia^2 \Bia^{2^*-1}  \big( \exp_{\xia}(\mia y)\big)\\
 & \lesssim \sa \mia^{1 - \frac{n}{2}}.
\end{aligned} \]
The previous arguments therefore show that 
 \ben \label{proplin17}
  \begin{aligned}
  \mia^2 \big| \ga R_\alpha &\big( \exp_{\xia}(\mia y)\big) \big|  \lesssim  \big( \tau_\alpha + \ea  +  \sa \gamma_\alpha  \big) \mia^{1 - \frac{n}{2}} 
  \end{aligned}
  \een
for some sequence $(\sa)_\alpha$ satisfying \eqref{defsigmaa}. Finally by \eqref{calcul4} and \eqref{proplin141} one has 
\ben \label{proplin18}
\bal \mia^2 \big|\mathcal{W}_\alpha\big( \exp_{\xia}(\mia y)\big) \big|^{2^*-2} |\vpa\big( \exp_{\xia}(\mia y)\big)| \\
\lesssim  (1+|y|)^{-4} \Vert \vpa \Vert_{L^\infty(\Oia)}
 \eal 
 \een
 for any $y \in \hat{\Omega}_{i,\alpha}$.  Combining \eqref{proplin1411}, \eqref{proplin15}, \eqref{proplin17} and \eqref{proplin18} in \eqref{eqvpa} yields \eqref{proplin14} and concludes the proof of the lemma.
 \end{proof}
 
 \medskip
 
 \textbf{Step $5$: control of $ \Vert \vpa \Vert_{L^\infty(\Oia)}$ and end of the proof of Proposition \ref{recurrence}.} 
 
The following result provides a control on $ \Vert \vpa \Vert_{L^\infty(\Oia)}$:
\begin{lemme} \label{lemmeestnorme}
There exists a sequence $(\sa^{\ell,5})_\alpha$ satisfying \eqref{defsigmaa} such that
\ben \label{proplin19}  \Vert \vpa \Vert_{L^\infty(\Oia)} \lesssim  \big( \tau_\alpha + \ea  +  \sa^{\ell,5}( \gamma_\alpha + a_\alpha)  \big)\mia^{1 - \frac{n}{2}}
\een
for all $\alpha \ge 1$.
\end{lemme}

\begin{proof}[Proof of Lemma \ref{lemmeestnorme}]
Let 
\ben \label{proplin191}
 \sa^{\ell,5} = \sa^{\ell,1} + \sa^{\ell,2} + \sa^{\ell,3} + \sa^{\ell,4} + \mu_{1, \alpha} + \sum_{i = 1}^k \left( \frac{\mia}{\ria}\right)^{n-2} + \sum \limits_{\underset{\mia = o(\mja) }{i,j \in \{1, \dots, k\}}} \left( \frac{\mia}{\mja} \right)^{\pui},
 \een
where the sequences $(\sa^{\ell,1})_\alpha, (\sa^{\ell,2})_\alpha, (\sa^{\ell,3})_\alpha$ and $(\sa^{\ell,4})_\alpha$ appear respectively in \eqref{proplin8}, \eqref{proplin821}, \eqref{proplin13} and \eqref{proplin14}. The sequence $(\sa^{\ell,5})$ obviously still satisfies \eqref{defsigmaa}. We claim that \eqref{proplin19} holds true for this choice of $\sa^{\ell,5}$. We proceed by contradiction and assume that, up to a subsequence,
\ben \label{proplin20}
 \Vert \vpa \Vert_{L^\infty(\Oia)} >>  \big( \tau_\alpha + \ea  +  \sa^{\ell,5}( \gamma_\alpha + a_\alpha)  \big)\mia^{1 - \frac{n}{2}}
\een 
where $\sa^{\ell,5}$ is given by \eqref{proplin191}. Our goal will be to pass \eqref{proplin14} to the limit and to get a contradiction with the assumption $\vpa \in K_\alpha^{\perp}$. 

We first claim that 
\ben \label{proplin21}
\sum_{j=1}^k (\Tja \Bja)\big( \exp_{\xia}(\mia y)\big) \lesssim \sa^{\ell,5} \mia^{1 - \frac{n}{2}} 
\een
for all $y \in \hat{\Omega}_{i,\alpha}$. Let indeed $y \in \hat{\Omega}_{i,\alpha}$. We first have, by \eqref{defTi} and and \eqref{proplin191},
\[ (\Tia \Bia) \big( \exp_{\xia}(\mia y)\big) \lesssim \mia^{2 - \frac{n}{2}} (1+|y|)^{3-n} \lesssim \sa^{\ell,5} \mia^{1- \frac{n}{2}}, \]
where we used that $\mia \le \mu_{1,\alpha}$ for all $i \ge 1$. Let now $j \neq i$. Assume first that $\mia = O(\mja)$, so that $j \in \mathcal{A}_i$. By Lemma \ref{compbulles} we have $\Bja\big( \exp_{\xia}(\mia y)\big)\lesssim \mia^{\pui}{\ria^{n-2}}$ for all $y\in \hat{\Omega}_{i,\alpha}$. Since $\Tja \lesssim 1$ this gives 
\[ (\Tja \Bja)\big( \exp_{\xia}(\mia y)\big) \lesssim \left( \frac{\mia}{\ria}\right)^{n-2} \mia^{1 - \frac{n}{2}} \lesssim \sa^{\ell,5} \mia^{1- \frac{n}{2}}.\]
Assume now that $\mja = o(\mia)$. As before, this either means that $d_g(\xia, \xja) > \frac{2 \ria}{R_0}$ or that $j \in \mathcal{B}_i$. In the first case, since $\Tja \lesssim 1$, we obtain that
\[(\Tja \Bja)\big( \exp_{\xia}(\mia y)\big) \lesssim    \left( \frac{\mja}{\mia} \right)^{\pui} \left(\frac{\mia}{\ria} \right)^{n-2} \mia^{1 - \frac{n}{2}}  \lesssim \sa^{\ell,5} \mia^{1- \frac{n}{2}}.\]
Assume finally that $j \in \mathcal{B}_i$. If $y \in \hat{\Omega}_{i,\alpha} \cap B(\hat{x}_{j,\alpha} , 1)$ then $d_g \big( \exp_{\xia}(\mia y), \xja\big) \lesssim \mia$. Hence, by \eqref{calcul},
\[  (\Tja \Bja)\big( \exp_{\xia}(\mia y)\big) \lesssim \mia^{2 - \frac{n}{2}}  \lesssim \sa^{\ell,5} \mia^{1- \frac{n}{2}}. \]
 If instead $y \in \hat{\Omega}_{i,\alpha} \cap \{ y \in \hat{\Omega}_{i,\alpha}, |y - {x}_{j,\alpha}| > 1 \}$ then $d_g \big( \exp_{\xia}(\mia y), \xja\big) \gtrsim \mia$ and hence, since $\Tja \lesssim 1$,
\[  (\Tja \Bja)\big( \exp_{\xia}(\mia y)\big) \lesssim  \left( \frac{\mja}{\mia} \right)^{\pui} \mia^{1 - \frac{n}{2}}  \lesssim \sa^{\ell,5} \mia^{1- \frac{n}{2}}.\]
Gathering these estimates proves \eqref{proplin21}. Let now, for $y \in \hat{\Omega}_{i,\alpha}$ 
\[ \psia(y) = \frac{1}{ \Vert \vpa \Vert_{L^\infty(\Oia)}} \hvpa(y) \]
where $\hvpa$ was defined in \eqref{deftvpa}. With  \eqref{proplin8},   \eqref{proplin191}, \eqref{proplin20}  and \eqref{proplin21} we obtain that $\psia$ satisfies
\ben \label{proplin22}
|\psia(y)| \lesssim (1+|y|)^{2-n} + o(1) \quad \textrm{ for all } y \in \hat{\Omega}_{i,\alpha}. 
 \een
Independently, a straightforward adaptation of the arguments that led to \eqref{proplin21} shows that for all $y \in \hat{\Omega}_{i,\alpha}$,
\[ 
\sum_{j=1}^k (\Xi_{j,\alpha} \Bja)\big( \exp_{\xia}(\mia y)\big) \lesssim  \mia^{1 - \frac{n}{2}} \sa^{\ell,5}
\]
holds, where $\Xi_{j,\alpha}$ is as in \eqref{defXia}. Combining the latter with  \eqref{proplin821}, \eqref{proplin191} and \eqref{proplin20} we also obtain that
 \ben \label{proplin221}
 |\nabla \psia(y)| \lesssim (1+|y|)^{1-n} + o(1)
 \een
for all $y \in \hat{\Omega}_{i,\alpha}$. If $y \in  \hat{\Omega}_{i,\alpha}$, $\psia$ also satisfies by \eqref{proplin14}
 \[|\triangle_{g_{i,\alpha}}\psia(y)| \lesssim 1. \]
Since $\hat{\Omega}_{i,\alpha}$ converges towards $\R^n\backslash \{(\hat{x}_{i,\infty})_{\in \mathcal{C}_i} \}$, standard elliptic theory and \eqref{proplin22} show that $\psia$ converges in $C^{1,\beta}_{loc}(\R^n\backslash \{(\hat{x}_{i,\infty})_{\in \mathcal{C}_i}\})$, for some $0 < \beta < 1$, towards $\psi_{\infty}$, where $\mathcal{C}_i$ is as in \eqref{defCi}.  By passing \eqref{proplin22} and \eqref{proplin221} to the limit we see that $\psi_{\infty}$ satisfies
\ben \label{proplin23}
|\psi_\infty(y)| + (1+|y|)|\nabla \psi_\infty(y)| \lesssim (1+|y|)^{2-n} \quad \textrm{ for all } y \in \R^n\backslash \{(\hat{x}_{i,\infty})_{\in \mathcal{C}_i}\}. 
 \een
Standard arguments then show that $\psi_\infty \in D^{1,2}(\R^n)$. With \eqref{proplin20}, \eqref{proplin221} and the dominated convergence theorem, we can also pass \eqref{proplin13} to the limit to obtain
\[
\int_{\R^n} \langle \nabla \psi_\infty, \nabla Z_{i,m} \rangle dx = 0  
\]
for any $1 \le m \le n_i$. Hence, $\psi_\infty \in K_{V_i}^{\perp}$ where $K_{V_i}$ is as in \eqref{defKV}. 

\medskip

We will now show that $\psi_\infty \in K_{V_i}$ and $\psi_\infty \not \equiv 0$ which will yield the desired contradiction. Let $y \in  \R^n\backslash \{(\hat{x}_{i,\infty})_{\in \mathcal{C}_i}\}$ be fixed. We first claim that, if $j \neq i$, 
\[ \mia^2 \Bja \big( \exp_{\xia}( \mia y) \big)^{2^*-2} \to 0 \quad \textrm{ as } \alpha \to + \infty. \]
This easily follows from the arguments developed in the proof of \eqref{proplin21}. With  \eqref{defWa} we then have that 
\[\frac{1}{\Vert \vpa \Vert_{L^\infty(\Oia)}} \mia^2\Big( |\mathcal{W}_\alpha|^{2^*-2} \vpa \Big) \big( \exp_{\xia}(\mia y)\big) \to V_i(y)^{2^*-2} \psi_{\infty}(y) \]
for all $y \in \hat{\Omega}_{i,\alpha}$ as $\alpha \to + \infty$. Using \eqref{proplin1411}, \eqref{proplin15}, \eqref{proplin17} and \eqref{proplin20} we can then pass the equation satisfied by $\psa$ to the limit to see that $\psi_\infty$ satisfies 
\[\triangle_\xi \psi_{\infty} = (2^*-1) |V_i|^{2^*-2}\psi_{\infty} \quad \textrm{ in } \R^n\backslash \{(\hat{x}_{i,\infty})_{\in \mathcal{C}_i} \} .\]
By \eqref{proplin23} the singularities of $\psi_\infty$ are removable, and $\psi_\infty$ extends to a $D^{1,2}(\R^n) \cap C^{1, \beta}(\R^n)$ function in $\R^n$, still denoted by $\psi_{\infty}$, satisfying 
\[ \triangle_\xi \psi_{\infty} = (2^*-1) |V_i|^{2^*-2}\psi_{\infty} \quad \textrm{ in } \R^n.
\]
This proves that  $\psi_\infty \in K_{V_i}$ and hence that $\psi_\infty \equiv 0$. It remains to see that $\psi_\infty \not \equiv 0$. By definition $\Vert \psia \Vert_{L^\infty(\hat{\Omega}_{i,\alpha})} = 1$ holds. If we let, for all $\alpha$, $y_\alpha \in \hat{\Omega}_{i,\alpha}$ be such that $|\psia(y_\alpha)| = 1$, \eqref{proplin22} shows that $|y_{\alpha}| \lesssim 1$ and hence, up to a subsequence, that $y_\alpha \to y_{\infty} \in \R^n$. If $y_\infty \in \R^n \backslash \{(\hat{x}_{i,\infty})_{\in \mathcal{C}_i} \}$ the local $C^{1, \beta}$ convergence shows that $|\psi_\infty(y_\infty)| =1$, a contradiction. If instead $y_\alpha \to \hat{x}_{j,\infty}$ for some $j \in \mathcal{C}_i$ we can find, by \eqref{proplin221}, another sequence of points of $\hat{\Omega}_{i,\alpha}$, that we call $(w_\alpha)_{\alpha}$, such that $w_\alpha \to w \in  \R^n \backslash \{(\hat{x}_{i,\infty})_{\in \mathcal{C}_i} \}$ and $\liminf_{\alpha \to + \infty} |\psia(w_\alpha)| >0$. This proves again that $\psi_\infty \not \equiv 0$, contradicts \eqref{proplin20} and concludes the proof of Lemma \ref{lemmeestnorme}.
\end{proof}

The remark after \eqref{proplin6} explained how the ``neck regions'' 
\[ B_g(\xja, \rjia) \backslash \cup_{s \in \mathcal{B}_i} B_g(x_{s, \alpha}, \frac{r_{s,\alpha}}{R_0}) \quad j \in \mathcal{B}_i, \] 
offer a negligible contribution to the pointwise estimates on $\vpa$, hence did not need to be taken into account in the induction property \eqref{Hl1}. The choice of the radius $\rjia$ might seem odd at first -- a more natural choice in view of the pointwise relations between bubbles would have been to choose the radius $\sjia$ as in \eqref{defsij} -- but is crucial in order to ensure the $C^1$-boundedness of $\psi_\infty$ in the contradiction argument of Lemma \ref{lemmeestnorme} (see \eqref{proplin23}).

\begin{proof}[End of the proof of Proposition \ref{recurrence}]
We are now in position to conclude the proof of the induction argument. We plug \eqref{proplin19} into \eqref{proplin8}, \eqref{proplin9}, \eqref{proplin821} and \eqref{proplin822}. This shows that, for any sequence $(\xa)_\alpha  \in B_g(\xia, \frac{ \ria}{R_0})$:
\[  
\begin{aligned}
& |\vpa(\xa)|  \lesssim  a_\alpha \frac{\mia^{\pui}}{\ria^{n-2}} +\ga \sum_{j=1}^k \Tja(\xa) \Bja(\xa) \\
&+ \big(\tau_\alpha + \ea + \sa^{\ell,5} (\ga +  a_\alpha) \big)\Big( \Bia(\xa) +  \sum_{j \in \mathcal{B}_i } \Bja(\xa) \Big), \\
 \end{aligned}
\]
where $\Tja$ is as in \eqref{defTi}, that
\ben \label{proplin231}
\begin{aligned}
& \int_{B_g(\xia, \frac{3 \ria}{2 R_0})} d_g(\xa, \cdot)^{2-n}\Big( \sum_{i=0}^k \Bia \Big)^{2^*-2}|\vpa| dv_g \\
& \lesssim  \big(\tau_\alpha + \ea + \sa^{\ell,5} (\ga +  a_\alpha) \big)\Big( \Bia(\xa) +  \sum_{j \in \mathcal{B}_i } \Bja(\xa) \Big), \\
& \int_{B_g(\xia, \frac{3\ria}{2 R_0})} d_g(\xa, \cdot)^{1-n}\Big( \sum_{i=0}^k \Bia \Big)^{2^*-2}|\vpa| dv_g \\
& \lesssim  \big(\tau_\alpha + \ea + \sa^{\ell,5} (\ga +  a_\alpha) \big) \Big( B_{0, \alpha}(\xa) +  \sum_{j =1}^k  \tja(\xa)^{-1}  \Bja(\xa) \Big),
\end{aligned}
\een
where $\Xi_{j,\alpha}$ is as in \eqref{defXia} and that 
\ben \label{proplin24} 
\begin{aligned}
& |\nabla \vpa(\xa)|  \lesssim  a_\alpha \frac{\mia^{\pui}}{\ria^{n-1}} +\gamma_\alpha \sum_{j=1}^k \Xi_{j,\alpha}(\xa) \Bja(\xa) \\
&+ \big(\tau_\alpha + \ea + \sa^{\ell,5} (\ga +  a_\alpha) \big) \Big( B_{0, \alpha}(\xa) +  \sum_{j =1}^k  \tja(\xa)^{-1}  \Bja(\xa) \Big), \\
 \end{aligned}
\een
where the sequence $\sa^{\ell,5}$ is as in \eqref{proplin191}. The first inequality remains true for all $i \in E_\ell$, and this together with $(H_{\ell+1,1})$ (in case $\ell \le p-1$) proves \eqref{Hl1}. To conclude the proof of \eqref{Hl2} it remains to prove that \eqref{proplin231} remains true for for any sequence $(\xa)_\alpha$ of points of $M$. First, it is easily seen that \eqref{proplin4}, \eqref{proplin5} and \eqref{proplin6} remain true for any sequence $(\xa)_\alpha$ of points of $M$, and not only of $B_g(\xia, \frac{\ria}{R_0})$ (remark that \eqref{proplin5} follows from $(H_{\ell+1,2})$ in case $\ell \le p-1$). If $(\xa)_\alpha$ is any sequence of points in $M$ we now have, using \eqref{Hl1}, that 
\[ \bal
\int_{B_g(\xia, \frac{\ria}{R_0})} d_g(\xa, \cdot)^{2-n} \Big( \sum_{i=0}^k \Bia \Big)^{2^*-2} |\vpa| dv_g \\
\lesssim  \big(\tau_\alpha + \ea + \sa^{\ell,5} (\ga +  a_\alpha) \big) \Big(  \Bia(\xa)+  \sum_{j \in \mathcal{B}_i}  \Bja(\xa) \Big). 
\eal\]

The second integral in \eqref{proplin231} is computed in the same way and \eqref{proplin231} remains true for any sequence $(\xa)_\alpha$  in $M$. This proves \eqref{Hl2} and concludes the proof of Proposition \ref{recurrence}.
\end{proof}

\subsection{Proof of Theorem \ref{proplin} part $2$: end of the proof} \label{proof3}

We are now in position to end the proof of Theorem \ref{proplin}. Let $G_0$ be the Green's function of the operator $\triangle_g + h - f'(u_0)$ in $M$, where  $f$ is as in \eqref{deff}. $G_0$ is by definition orthogonal to $K_0$, which is defined in \eqref{defK0}, and there exists $C>0$ depending only on $n, (M,g), h$ and $u_0$ so that $|G_0(x,y)| \le C d_g(x,y)^{2-n}$ for any $x \neq y$ in $M$ (see Robert \cite{RobDirichlet}). By \eqref{eqvpa} $\vpa$ satisfies
\[ \begin{aligned}
\triangle_g \vpa + h \vpa - f'(u_0)  \vpa &= \big( f'(\mWa) - f'(u_0) \big) \vpa +  (h - h_\alpha )\vpa  \\
&+  R_\alpha +  \sum \limits_{\underset{j =1 .. n_i }{i=0..k}} \lija ( \triangle_g   + 1)Z_{i,j, \alpha}.
 \end{aligned}
 \]
Let $(\xa)_\alpha$ be a sequence of points in $M$. By definition $\vpa \in K_0^{\perp}$ and so a representation formula for $\triangle_g + h - f'(u_0)$  in $M$ shows, using \eqref{proplin2}, \eqref{defaa}, \eqref{proplin1} and \eqref{greenB} below, that there exists a sequence $(\sa)_\alpha$ satisfying \eqref{defsigmaa} such that 
\ben \label{proplin25}
\begin{aligned}
|\vpa(\xa)| & \lesssim \big( \tau_\alpha + \ea + \sa \gamma_\alpha \big)  \sum_{j=0}^k \Bja(\xa) \\
&  +\gamma_\alpha\sum_{j=1}^k \Tja(\xa) \Bja(\xa)  \\
& + \int_{M} d_g(\xa, \cdot)^{2-n} \big|  f'(\mWa) - f'(u_0) \big| |\vpa| dv_g.
\end{aligned}
\een
By \eqref{calcul3} we have 
\[ |f'(\mWa) - f'(u_0)| \lesssim \Big(\sum_{i=1}^k \Bia\Big)^{2^*-2} \quad \textrm{ in } \bigcup_{i=1}^k  B_g(\xia, \frac{\ria}{R_0}). \]
Proposition \ref{recurrence} (more precisely \eqref{Hl2}) then shows that 
\ben \label{proplin26}
\begin{aligned}
 \int_{ \bigcup_{i=1}^k  B_g(\xia, \frac{\ria}{R_0})} &d_g(\xa, \cdot)^{2-n} \big|  f'(\mWa) - f'(u_0) \big| |\vpa| dv_g \\
 & \lesssim \big( \tau_\alpha + \ea + \sa (\gamma_\alpha+a_\alpha) \big)   \sum_{j=0}^k \Bja(\xa).
\end{aligned}
\een
To estimate the integral over $M \backslash  \bigcup_{i=1}^k  B_g(\xia, \frac{\ria}{R_0})$ we distinguish two cases. Assume first that $B_{0,\alpha} \equiv 0$. According to \eqref{defBia} this corresponds to the case where $u_0 \equiv 0$ and $\triangle_g +h$ has no kernel. Then $f'(u_0) \equiv 0$ and on  $M \backslash  \bigcup_{i=1}^k  B_g(\xia, \frac{\ria}{R_0})$ we have 
\[|f'(\mWa)| \lesssim \sum_{i=1}^k \Bia^{2^*-2}. \]
Straightforward computations with \eqref{defaa} and \eqref{decbulle} below then show that 
\ben \label{proplin27}
\begin{aligned}
& \int_{ M \backslash \bigcup_{i=1}^k  B_g(\xia, \frac{\ria}{R_0})} d_g(\xa, \cdot)^{2-n} \big|  f'(\mWa) \big| |\vpa| dv_g \\
 & \lesssim  a_{\alpha} \sum_{j=1}^k \left( \frac{\mja}{\rja}\right)^2 \Bja(\xa) = \sa a_\alpha \sum_{j=1}^k \Bia(\xa)
\end{aligned}
\een
for some sequence $(\sa)_\alpha$ satisfying \eqref{defsigmaa}. Assume now that $B_{0,\alpha} \equiv 1$, which happens when $u_0 \not \equiv 0$ or when $u_0 \equiv 0$ and $\triangle_g +h$ has kernel. By \eqref{defri} and \eqref{defR0} $\ria \le \sqrt{\mia}$ holds true for all $1 \le i \le k$ and we can decompose $M \backslash  \bigcup_{i=1}^k  B_g(\xia, \frac{\ria}{R_0})$ as
\[ \begin{aligned}
 &M \backslash  \bigcup_{i=1}^k  B_g(\xia, \frac{\ria}{R_0}) \subseteq \\
 &  \Big( M \backslash  \bigcup_{i=1}^k  B_g(\xia, \sqrt{\mia}) \Big)\bigcup  \bigcup_{i=1}^k \Big( B_g(\xia, \sqrt{\mia}) \backslash \bigcup_{j=1}^k B_g(\xja, \frac{\rja}{R_0}) \Big) .
  \end{aligned}
 \]
If $i \in \{1, \dots,k\}$ we have, by \eqref{defBia} and for all $y \in  B_g(\xia, \sqrt{\mia}) \backslash \bigcup_{j=1}^k B_g(\xja, \frac{\rja}{R_0})$,
\[  \big|  f'(\mWa) - f'(u_0) \big|(y)\lesssim \sum_{i=1}^k \Bia(y)^{2^*-2}. \]
As a consequence and by \eqref{defaa} and \eqref{decbulle} below, for any $i \in \{1, \dots, k\}$ we have
 \ben \label{proplin28}
\begin{aligned}
&  \int \limits_{ B_g(\xia, \sqrt{\mia}) \backslash \bigcup_{j=1}^k B_g(\xja, \frac{\rja}{R_0})} d_g(\xa, \cdot)^{2-n} \big|  f'(\mWa) - f'(u_0) \big| |\vpa| dv_g \\
& \lesssim   \int \limits_{ B_g(\xia, \sqrt{\mia}) \backslash B_g(\xia, \frac{\ria}{R_0})} a_\alpha d_g(\xa, \cdot)^{2-n} \Bia^{2^*-1} dv_g \\ 
 & \lesssim  \sa a_{\alpha} \sum_{j=1}^k \Bja(\xa).
\end{aligned}
\een
Since $B_{0,\alpha} = 1$ we have 
\ben \label{proplin281}
 \Tia(y) \Bia(y) \lesssim \sa(\Bia(y) + B_{0,\alpha}(y) ) 
 \een
for all $y \in M$. This follows e.g. from \eqref{defTi} and the simple inequality
\[  \Tia \Bia \lesssim \sqrt{\mia} \Bia \mathds{1}_{\{ d_g(\xia, \cdot) \le \sqrt{\mia}\} } + \sqrt{\mia}\mathds{1}_{\{ d_g(\xia, \cdot) \ge \sqrt{\mia} \}}. \]
If $y \in M \backslash  \bigcup_{i=1}^k  B_g(\xia, \sqrt{\mia}) $ we then have 
\[  \big|  f'(\mWa) - f'(u_0) \big|(y) \lesssim \sum_{i=1}^k \Bia(y) + \sum_{i=1}^k \Bia(y)^{2^*-2} + \sa B_{0,\alpha} \]
for some sequence $(\sa)_\alpha$ satisfying \eqref{defsigmaa}, so that \eqref{proplin281}, \eqref{greenB} and \eqref{decbulle} below give
\ben \label{proplin29}
\begin{aligned}
 \int_{M \backslash  \bigcup_{i=1}^k  B_g(\xia, \sqrt{\mia})} &d_g(\xa, \cdot)^{2-n} \big|  f'(\mWa) - f'(u_0) \big| |\vpa| dv_g \\
 & \lesssim  \sa a_{\alpha} \sum_{j=0}^k \Bja(\xa).
\end{aligned}
\een
By plugging \eqref{proplin26}, \eqref{proplin27},  \eqref{proplin28} and \eqref{proplin29} in  \eqref{proplin25} we have thus proven that there exists a sequence $(\sa)_\alpha$ satisfying \eqref{defsigmaa} such that 
\ben \label{proplin30}
\begin{aligned}
|\vpa(\xa)| & \lesssim \big( \tau_\alpha + \ea + \sa (\gamma_\alpha+a_\alpha) \big)   \sum_{j=0}^k \Bja(\xa) \\
&  +\gamma_\alpha  \sum_{j=1}^k \Tja(\xa) \Bja(\xa) 
\end{aligned}
\een
for any sequence $(\xa)_\alpha$ in $M$.  Estimates on $|\nabla \vpa|$ are obtained in a similar way, just as we did in \eqref{proplin821}:  a straightforward adaptation of the arguments that led to \eqref{proplin26}, \eqref{proplin27},  \eqref{proplin28} and \eqref{proplin29} then shows that
\ben \label{proplin31}
\begin{aligned}
& |\nabla \vpa(\xa)| \\
&  \lesssim\big( \tau_\alpha + \ea + \sa (\gamma_\alpha+a_\alpha) \big)  \big( B_{0,\alpha} + \sum_{j=1}^k \tja(\xa)^{-1}\Bja(\xa)\big) \\
&  +\gamma_\alpha  \sum_{j=1}^k \Xi_{j,\alpha}(\xa) \Bja(\xa) 
\end{aligned}
\een
for the same sequence $(\sa)_\alpha$ as in \eqref{proplin30} and for any sequence $(\xa)_\alpha$ of points of $M$. Going back to the definition of $a_\alpha$ in \eqref{defaa} we let $(y_\alpha)_\alpha$ be a sequence of points of $M$ such that
\[ 
 \left| \frac{\vpa(y_\alpha)}{\sum_{i=0}^k \Bia(y_\alpha)}\right | + \left|  \frac{\nabla \vpa(y_\alpha)}{B_{0,\alpha}(y_\alpha) + \sum_{i=1}^k \tia(y_\alpha)^{-1} \Bia(y_\alpha)}\right| = a_\alpha.
\]
Since $\Tja \lesssim 1$ and $\Xi_{j,\alpha} \lesssim 1$ for all $1 \le j \le k$ we have $ \sum_{j=1}^k \Tja \Bja \lesssim \sum_{i=1}^k \Bia$ and $ \sum_{j=1}^k \Xi_{j,\alpha} \Bja\lesssim  \sum_{j=1}^k \tja^{-1}\Bja$ in $M$. Applying \eqref{proplin30} and \eqref{proplin31} at $\xa=y_\alpha$ then yields  
\ben \label{soussuite}
 a_\alpha \lesssim \tau_\alpha + \ea + \gamma_\alpha + \sa a_\alpha. 
 \een
Since $\sa \to 0$ by \eqref{defsigmaa}, \emph{up to passing to a subsequence for the sequence $(\sa)_\alpha$ only} (which amounts to passing to a common subsequence for $(\xia)_\alpha, (\mia)_\alpha$, $1 \le i \le k$, and $(h_\alpha)_\alpha$), the previous inequality self-improves and gives 
\[ a_\alpha \lesssim  \tau_\alpha + \ea + \gamma_\alpha. \]
Replugging the latter in \eqref{proplin30} finally gives 
\[
\begin{aligned}
|\vpa(\xa)| & \lesssim \big( \tau_\alpha + \ea + \sa \gamma_\alpha  \big) \sum_{j=0}^k \Bja(\xa) +\gamma_\alpha \sum_{j=1}^k \Tja(\xa) \Bja(\xa) 
\end{aligned}
\]
for any sequence $(\xa)_\alpha$ in $M$, and concludes the proof of  Theorem \ref{proplin}. \vspace{-15pt} \begin{flushright} $\square$ \end{flushright}

\begin{remark}
The proof of Theorem \ref{proplin} keeps track of all the sequences appearing in the estimates by distinguishing the error terms coming from \eqref{propRa} from those that originate from the interaction between the bubbles $\Via$ themselves. It is particularly important to single out, in the estimates, the sequences satisfying \eqref{defsigmaa}, i.e. that only depend on the choice of $(\xia)_\alpha, (\mia)_\alpha$, $1 \le i \le k$, and $(h_\alpha)_\alpha$. See for instance the final self-improving argument of \eqref{soussuite} that proves that, up to passing to a subsequence for the sequences $(\xia)_\alpha, (\mia)_\alpha$, $1 \le i \le k$, and $(h_\alpha)_\alpha$ only, estimate \eqref{estopt} holds true for any of the sequences $(\ta)_\alpha, (\ea)_\alpha$ and $(\ga)_\alpha$. This will be used again in the proof of Proposition \ref{nonlin2} below. 
\end{remark}

  \begin{remark}
Theorem \ref{proplin} remains true for a slighty more general choice of $K_\alpha$ in \eqref{defKa}. A careful inspection of the proof shows for instance that when $i \ge 1$ one can replace the $Z_{i,j,\alpha}$ explicitly given by \eqref{defZij} by another family $(Z_{i,j,\alpha})_\alpha$ of smooth functions in $M$, $1 \le j \le  n_i$, satisfying 
\[ \big|(\triangle_g +1)Z_{i,j,\alpha} \big| \lesssim \Bia^{2^*-1} + \Bia, \quad |Z_{i,j,\alpha}| \lesssim \Bia,  \quad |\nabla Z_{i,j,\alpha}| \lesssim \tia^{-1} \Bia \textrm{ and } \]
\[ \mia^{\pui} Z_{i,j,\alpha} \big( \exp_{\xia}(\mia x) \big) \to Z_{i,j}(x) \textrm{ in } C^0_{loc}(\R^n). \]
  \end{remark}

\section{Nonlinear procedure and proof of Theorem \ref{theorieC0}} \label{nonlinear}

In this section we prove Theorem \ref{theorieC0}. We first perform a nonlinear perturbative argument to show that \eqref{eqha} possesses, up to kernel elements, a canonical solution that looks like $\mWa$ at first-order and that comes with explicit pointwise bounds. This is the content of Proposition \ref{nonlin2}, whose proof heavily relies on Theorem \ref{proplin}. We then apply this construction to obtain \emph{a priori} bounds and prove a slightly more general result, Theorem \ref{theopsC0} below, from which Theorem \ref{theorieC0} follows.

\subsection{Nonlinear procedure}

As in Section \ref{linear} we let $(h_\alpha)_\alpha$ be a sequence of functions converging towards $h$ in $L^\infty(M)$ as $\alpha \to + \infty$, $k \ge 1$ be an integer, $u_0 \in C^{1,\beta}(M)$ for some $0 < \beta < 1$ solve 
\[\triangle_g u_0 + h u_0 = |u_0|^{2^*-2} u_0, \]
 $(\xia)_{\alpha}$ and $(\mia)_\alpha$, $1 \le i \le k$ be $k$ sequences satisfying \eqref{structure}, we let $V_1, \dots, V_k $ in $D^{1,2}(\R^n)$ be $k$ solutions of \eqref{yamabe} and we let $(\mWa)_\alpha$ be a sequence of continuous functions satisfying \eqref{defWa}. We assume in addition that there exists a positive constant $C_1$, only depending on $n, (M,g), V_1, \dots,  V_k, u_0$ and $h$, and a sequence $(\vea)_\alpha \in \mathcal{E}$  such that for all $\alpha \ge 1$
\ben \label{errWa2}
\bal
\Big| \triangle_g \mWa + h_\alpha \mWa - &|\mWa|^{2^*-2} \mWa \Big| \le C_1(\vea+\Vert h_\alpha - h \Vert_{L^\infty(M)}) B_{0,\alpha} \\
&+ C_1\sum_{i=1}^k \Bia +  C_1 \sum \limits_{\underset{i \neq j}{i,j=0..k}}
 \Bia^{2^*-2} \Bja + C_1\vea \sum_{i=1}^k \Bia^{2^*-1}
\eal
\een
in $M$.  We keep in this section the notations of Section \ref{linear}, in particular \eqref{ineqq} and \eqref{defsigmaa}. Our first result constructs a canonical solution of \eqref{eqha} in $H^1(M)$, up to kernel elements, that looks like $\mWa$ at first-order:

\begin{prop} \label{nonlin1}
Let $(\mWa)$ be as in \eqref{defWa} and \eqref{errWa2}. There exists $\delta >0$ such that, up to a passing to a subsequence for $(\xia)_\alpha, (\mia)_\alpha$, $1 \le i \le k$, $(h_\alpha)_\alpha$ and $(\vea)_\alpha$, there exists a unique solution 
\[\vpa \in K_\alpha^{\perp} \cap \{ \vp \in H^1(M), \Vert \vp \Vert_{H^1(M)} \le \delta \}\]
 of 
\[ \Pi_{K_\alpha^{\perp}} \Bigg( \mWa + \vpa - (\triangle_g + 1)^{-1} \Big[ (1 - h_\alpha)(\mWa + \vpa) + f\big( \mWa + \vpa \big)  \Big] \Bigg) = 0,\]
where $f$ is as in \eqref{deff}.
\end{prop}
This result is a simple variant of the standard nonlinear procedure found e.g. in Esposito-Pistoia-V\'etois \cite{EspositoPistoiaVetois} or Robert-V\'etois \cite{RobertVetois} and we omit its proof. We now prove that the linear theory of Theorem \ref{proplin} provides pointwise estimates on the remainder constructed in Proposition \ref{nonlin1}:

\begin{prop} \label{nonlin2}
Let $(h_\alpha)_\alpha$ be a sequence of functions converging towards $h$ in $L^\infty(M)$, $k \ge 1$ be an integer, $u_0 \in C^{1,\beta}(M)$ for $0 < \beta < 1$ be a solution of \eqref{limiteq}, $V_1, \dots, V_k $ in $D^{1,2}(\R^n)$ be $k$ solutions of \eqref{yamabe} and let $(\xia)_{\alpha}$ and $(\mia)_\alpha$, $1 \le i \le k$ be $k$ sequences satisfying \eqref{structure}. Let $(\mWa)_\alpha$ be a sequence of continuous functions in $M$ satisfying \eqref{defWa} and \eqref{errWa2} for some $(\vea)_\alpha \in \mathcal{E}$. There exists $C>0$ independent of $\alpha$ and there exists a sequence $(\sa)_\alpha$ satisfying \eqref{defsigmaa} such that, up to passing to a subsequence for $(\xia)_\alpha, (\mia)_\alpha$, $1 \le i \le k$, $(h_\alpha)_\alpha$ and $(\vea)_\alpha$, the equation
\ben \label{eqnonlin}
 \Pi_{K_\alpha^{\perp}} \Bigg( \mWa + \vpa - (\triangle_g + 1)^{-1} \Big[ (1 - h_\alpha)(\mWa + \vpa) + f\big( \mWa + \vpa \big)  \Big] \Bigg) = 0
 \een
has a unique solution $\vpa$ in
\[ \bal K_\alpha^{\perp} \cap \Bigg \{ \vp \in C^0(M), \quad &|\vp(x)| \le C(\vea + \sa)^{\frac12} \sum_{i=0}^k \Bia(x) \\
&+ C \sum_{i=1}^k \Tia(x) \Bia(x)  \quad \textrm{ for all } x \in M \Bigg \} . \eal \]
 This solution $\vpa$ also satisfies
\ben \label{normereste}
\Vert \vpa \Vert_{H^1(M)} \le C(\vea + \sa).
\een
\end{prop}
Recall that $\mathcal{E}$ is as in \eqref{defE}, $\Tia$ is as in \eqref{defTi} and $f$ is as in \eqref{deff}.

\begin{proof}
Let $(h_\alpha)_\alpha$, $u_0$, $V_1, \dots, V_k $, $(\xia)_{\alpha}$ and $(\mia)_\alpha$, $1 \le i \le k$ be as in the statement of Proposition \ref{nonlin2} and let $(\mWa)_\alpha$ be a sequence of continuous functions satisfying \eqref{defWa} and \eqref{errWa2}. Throughout this proof we fix a real number $0 < \tau < 2^*-2$. We claim that there exists a sequence $(\sa^1)_{\alpha}$ satisfying \eqref{defsigmaa} such that for all $x \in M$ 
\ben  \label{prnonlin1}
\begin{aligned}
& \Big( \sum_{i=0}^k \Bia (x) \Big)^{2^*-1-\tau} \Big( \sum_{i=1}^k \Tia(x) \Bia(x) \Big)^{\tau} +  \Big( \sum_{i=1}^k \Tia(x) \Bia(x) \Big)^{2^*-1}    \\
& + \Big( \sum_{i=0}^k \Bia (x) \Big)^{1+\tau} \Big( \sum_{i=1}^k \Tia(x) \Bia(x) \Big)^{2^*-2-\tau} \\
& + \Big( \sum_{i=0}^k \Bia (x) \Big)^{2^*-2} \Big( \sum_{i=1}^k \Tia(x) \Bia(x) \Big) \\
& \le \sa^1 \Big( \sum_{i=0}^k \Bia(x) + \sum_{i=1}^k \Bia(x)^{2^*-1} \Big).
\end{aligned}
\een
This easily follows from Lemma \ref{lemtec} proven in the Appendix. Since $\mWa$ satisfies \eqref{defWa} Theorem \ref{proplin} applies and, up to passing to a subsequence for  $(\xia)_\alpha, (\mia)_\alpha$, $1 \le i \le k$ and $(h_\alpha)_\alpha$, shows the existence of $C_0>0$ and of a sequence $(\sa^2)_\alpha$ satisfying \eqref{defsigmaa} such that for any $(\ta)_\alpha, (\ea)_\alpha \in \mathcal{E}$, any bounded sequence $(\ga)_\alpha$ and any $(R_\alpha)_\alpha$ satisfying \eqref{propRa}, \eqref{estopt} holds with $\sa = \sa^2$. We now let $(\vea)_\alpha \in \mathcal{E}$ and assume in addition that $\mWa$ satisfies \eqref{errWa2}. Define the following sequence:
\ben \label{defnua}
  \nua =  \big(\sa^1 + \sa^2 + \Vert h_\alpha - h_0 \Vert_{L^\infty(M)} + \vea \big)^{\frac12}
  \een
and define 
\ben \label{defSa}
\bal 
&\mathcal{S}_\alpha = \Big \{ \vp \in C^0(M) \textrm{ such that } \\
&  |\vp(x)| \le 2C_0 C_1 \nua \sum_{i=0}^k \Bia(x) +2C_0C_1 \sum_{i=1}^k \Tia(x) \Bia(x) \textrm{ for all } x \in M  \\
& \textrm{ and } \int_M \vp (\triangle_g + 1) Z_{i,j,\alpha} dv_g = 0 \textrm{ for all } i=0 \dots k, j=1 \dots  n_i  \Big \},
\eal 
\een
where $C_1$ is the constant appearing in \eqref{errWa2}. We endow $\mathcal{S}_\alpha$ with the norm 
\[ \Vert \vp \Vert_* = \left \Vert \frac{\vp}{\nua\sum_{i=0}^k \Bia + \sum_{i=1}^k \Tia \Bia} \right \Vert_{L^\infty(M)}, \quad \vp \in \mathcal{S}_\alpha. \]
For any fixed value of  $\alpha$, $(\mathcal{S}_\alpha, \Vert \cdot \Vert_*)$ is a Banach space and it is easily seen that if $\vp \in \mathcal{S}_\alpha \cap C^1(M)$ then $\vp \in K_\alpha^{\perp}$ where $
K_\alpha$ is as in \eqref{defKa}. If $\vp \in C^0(M)$ we define $T_\alpha(\vp) \in K_\alpha^{\perp}$ as the unique solution  in $K_\alpha^{\perp}$ of
\[ \big(\triangle_g + h_\alpha - f'(\mWa) \big)T_\alpha(\vp) = \tilde{R}_\alpha + N_\alpha(\vp) +  \sum \limits_{\underset{j =1 .. n_i }{i=0..k}} \lija ( \triangle_g   + 1)Z_{i,j, \alpha} \]
for some unique real numbers $\lija$, where we have let  
\ben \label{Rapreuve}
\bal \tilde{R}_\alpha & = - \triangle_g \mWa - h_\alpha \mWa + |\mWa|^{2^*-2} \mWa \quad \textrm{ and }  \\
N_\alpha(\vp) & = f(\mWa+ \vp) - f(\mWa) - f'(\mWa) \vp. \eal
\een
The existence of $T_\alpha(\vp)$ is ensured by Proposition \ref{propH1}, and by standard elliptic theory $T_\alpha(\vp) \in C^1(M)$. To prove Proposition \ref{nonlin2} we will prove that, up to a subsequence, $T_\alpha$ is a contraction on $S_\alpha$.

\medskip

\textbf{Step $1$: $T_\alpha$ stabilizes $\mathcal{S}_\alpha$.} If $\vp \in \mathcal{S}_\alpha$ we have 
\[ |N_\alpha(\vp)| \lesssim \Big(\sum_{i=0}^k \Bia \Big)^{2^*-2-\tau}|\vp|^{1+\tau} + |\vp|^{2^*-1} \]
in $M$. By definition of $\mathcal{S}_\alpha$ and by \eqref{prnonlin1} there exists $C_2 >0$ such that
\[ 
\begin{aligned}
|N_\alpha(\vp)| \le C_2(\sa^1 + \nua^{1+\tau}) \sum_{i=1}^k \Bia ^{2^*-1} + C_2 \sa^1  \sum_{i=0}^k \Bia
\end{aligned}
\]
in $M$. Together with \eqref{errWa2} we thus have 
\[ \big| \tilde{R}_\alpha +  N_\alpha(\vp) \big| \le  \tau_\alpha B_{0,\alpha} +  \ea \sum_{i=1}^k \Bia^{2^*-1} + \gamma_\alpha \Big( \sum_{i=1}^k \Bia +   \sum \limits_{\underset{i \neq j}{i,j=0..k}} \Bia^{2^*-2} \Bja\Big) \]
in $M$, where we have let  
\ben \label{suites}
 \bal \tau_\alpha&  = C_1(\vea +  \Vert h_\alpha - h_0 \Vert_{L^\infty(M)}) + C_2 \sa^1, \\
\ea & = C_1 \vea + C_2 (\va^{1+\tau} + \sa^1) , \\
\ga & = C_1 + C_2 \sa^1 .\eal \een
A direct application of Theorem \ref{proplin} with $R_\alpha = \tilde{R}_\alpha +  N_\alpha(\vp)$ then shows that, for any $x \in M$,
\[\bal
&  |T_\alpha(\vp)(x)| \le C_0 \big( \ta + \ea + \sa^2 \gamma_\alpha   \big) \sum_{i=0}^k \Bia(x)  + C_0\gamma_\alpha \sum_{i=1}^k \Tia(x) B_i(x) 
\eal \]
holds with $\ta, \ea, \ga$ given by \eqref{suites}. Since $\sa^1 \to 0$ and since, by \eqref{defnua}, we have $ \ta + \ea + \sa^1 \gamma_\alpha = o(\nua)$ as $\alpha \to + \infty$ we can assume, up to passing to a subsequence for $(\xia)_\alpha, (\mia)_\alpha$, $1 \le i \le k$, $(h_\alpha)_\alpha$ and $(\vea)_\alpha$, that 
 \[  \ta + \ea + \sa^1 \gamma_\alpha  \le 2 C_1 \nua \textrm{ and } \gamma_\alpha \le 2 C_1. \]
By definition of $\mathcal{S}_\alpha$ in \eqref{defSa}, this shows that $T_\alpha$ stabilises $\mathcal{S}_\alpha$.
 
 \medskip

 \textbf{Step $2$: $T_\alpha$ is a contraction on $\mathcal{S}_\alpha$.} Let $\vp_1, \vp_2 \in \mathcal{S}_\alpha$. By definition $T_\alpha(\vp_1)-  T_\alpha(\vp_2)$ satisfies
 \ben \label{ctrc} \bal
  \big(\triangle_g + h_\alpha &- f'(\mWa) \big)(T_\alpha(\vp_1) - T_\alpha(\vp_2))  = \sum \limits_{\underset{j =1 .. n_i }{i=0..k}} \tilde{\lambda}_{i,j,\alpha} ( \triangle_g   + 1)Z_{i,j, \alpha} \\
  & + f(\mWa + \vp_1) - f(\mWa + \vp_2) - f'(\mWa)(\vp_1 - \vp_2),
  \eal \een
for some real numbers $ \tilde{\lambda}_{i,j,\alpha} $. By the mean-value inequality we have
\[ \begin{aligned}
\big|  f(\mWa  + \vp_1)& - f(\mWa + \vp_2) - f'(\mWa)(\vp_1 - \vp_2) \big| \\
& \lesssim \Big(\sum_{i=0}^k \Bia \Big)^{2^*-2-\tau}(|\vp_1|^{\tau} + |\vp_2|^{\tau}) |\vp_1 - \vp_2|
\end{aligned} \]
in $M$. By \eqref{prnonlin1}, by definition of $\mathcal{S}_\alpha$ in \eqref{defSa} and of the norm $\Vert \cdot \Vert_*$ we get that there exists a positive constant $C_3 $ such that
\[ \begin{aligned}
& \big|  f(\mWa  + \vp_1) - f(\mWa + \vp_2) - f'(\mWa)(\vp_1 - \vp_2) \big| \\
& \le C_3 \Vert \vp_1  - \vp_2 \Vert_{*} (\nua^{1+\tau} + \nua^{\tau}\sa^1 + \nua \sa^1 + \sa^1) \Big( \sum_{i=1}^k \Bia^{2^*-1}  +  \sum_{i=0}^k \Bia \Big)
\end{aligned} \]
everywhere in $M$. Note, by \eqref{defSa}, that $ \Vert \vp_1  - \vp_2 \Vert_{*}$ is uniformly bounded. We thus have 
\[ \bal 
\big|    f(\mWa  + \vp_1)& - f(\mWa + \vp_2) - f'(\mWa)(\vp_1 - \vp_2) \big|\\
&  \le  \tau_\alpha B_{0,\alpha} +  \ea \sum_{i=1}^k \Bia^{2^*-1} + \gamma_\alpha \Big( \sum_{i=1}^k \Bia +   \sum \limits_{\underset{i \neq j}{i,j=0..k}} \Bia^{2^*-2} \Bja\Big)
\eal\]
in $M$, where we have let  
\[ \tau_\alpha  = \ea = \ga = C_3 \Vert \vp_1  - \vp_2 \Vert_{*}  \big( \nua^{1+\tau} + \nua^{\tau}\sa^1 + \nua \sa^1 + \sa^1\big).\]
A direct application of Theorem \ref{proplin} with $R_\alpha =   f(\mWa  + \vp_1) - f(\mWa + \vp_2) - f'(\mWa)(\vp_1 - \vp_2) $ then shows that
\[\bal
&  \big|   T_\alpha(\vp_1) - T_\alpha(\vp_2) \big| \le C_0 \big( \ta + \ea + \sa^2 \gamma_\alpha   \big) \sum_{i=0}^k \Bia  + C_0\gamma_\alpha \sum_{i=1}^k \Tia B_i
\eal \]
holds everywhere in $M$. We have $\ga = o(1)$ and 
\[ C_3 \big( \nua^{1+\tau} + \nua^{\tau}\sa^1 + \nua \sa^1 + \sa^1 \big) = o(\nua) \]
as $\alpha \to + \infty$.  Up to passing to a subsequence for $(\xia)_\alpha, (\mia)_\alpha$, $1 \le i \le k$, $(h_\alpha)_\alpha$ and $(\vea)_\alpha$ the latter inequality becomes, by definition of $\Vert \cdot \Vert_{*}$,
 \[  \begin{aligned}
 \Vert  T_\alpha(\vp_1) - T_\alpha(\vp_2) \Vert_* & \le \frac12 \Vert \vp_1 - \vp_2\Vert_*
 \end{aligned}  \]
and shows that $T_\alpha$ is a contraction in $S_\alpha$. 

\medskip

Picard's fixed-point theorem now applies to $T_\alpha$ on $(\mathcal{S}_\alpha, \Vert \cdot \Vert_*)$ and shows the existence of a unique $\vpa \in \mathcal{S}_\alpha$ such that $T_\alpha(\vpa) = \vpa$, that is such that  
\[ (\triangle_g + h_\alpha)\big( \mWa + \vpa \big) = f\big( \mWa + \vpa \big)  +  \sum \limits_{\underset{j =1 .. n_i }{i=0..k}} \lija ( \triangle_g   + 1)Z_{i,j, \alpha} \]
for some real numbers $\lija$. Elliptic regularity shows that $\vpa$ is of class $C^1$ and thus belongs to $K_\alpha^{\perp}$. Clearly, for a function $\vp \in \mathcal{S}_\alpha\cap K_\alpha^{\perp}$, solving the latter equation is equivalent to solving \eqref{eqnonlin}. It remains to prove \eqref{normereste} to conclude the proof of Proposition \ref{nonlin2}. The equation above rewrites as 
\[ \big(\triangle_g + h_\alpha - f'(\mWa) \big) \vpa= \tilde{R}_\alpha + N_\alpha(\vpa) +  \sum \limits_{\underset{j =1 .. n_i }{i=0..k}} \lija ( \triangle_g   + 1)Z_{i,j, \alpha}, \]
where $\tilde{R}_\alpha$ is as in \eqref{Rapreuve}. Since $\vpa \in \mathcal{S}_\alpha$ we have, by \eqref{defnua}, $\Vert \vpa \Vert_{L^{2^*}(M)} = o(1)$ and
\[ \Vert N_\alpha(\vpa) \Vert_{L^{\frac{2n}{n+2}}(M)} \lesssim\Vert \vpa \Vert_{L^{2^*}(M)}^{1+\tau}. \]
Independently, we also have with \eqref{errWa2} and \eqref{Rapreuve} that 
\[ \Vert \tilde{R}_\alpha \Vert_{L^{\frac{2n}{n+2}}(M)} \lesssim \vea + \sa. \]
Sobolev's inequality and Proposition \ref{propH1} then show that 
\[ \bal 
 \Vert \vpa \Vert_{L^{2^*}(M)} \lesssim \Vert \vpa \Vert_{H^1(M)} & \lesssim \Vert (\triangle_g + 1)^{-1}(  \tilde{R}_\alpha + N_\alpha(\vpa)) \Vert_{H^1(M)} \\
& \lesssim \vea + \sa + \Vert \vpa \Vert_{L^{2^*}(M)}^{1+\tau},
\eal \]
which proves \eqref{normereste}.
\end{proof}

\subsection{A priori estimates for solutions of \eqref{eqnonlin}}

We now prove a slightly more general result than Theorem \ref{theorieC0}, that states that solutions of \eqref{eqnonlin} inherit pointwise estimates:

\begin{theo} \label{theopsC0}
Let $(h_\alpha)_\alpha$ be a sequence of functions converging towards $h$ in $L^\infty(M)$, $k \ge 1$ be an integer, $u_0 \in C^{1,\beta}(M)$ for $0 < \beta < 1$ be a solution of \eqref{limiteq}, $V_1, \dots, V_k $ in $D^{1,2}(\R^n)$ be $k$ solutions of \eqref{yamabe} and let $(\xia)_{\alpha}$ and $(\mia)_\alpha$, $1 \le i \le k$ be $k$ sequences satisfying \eqref{structure}. Let $(\mWa)_\alpha$ be a sequence of continuous functions in $M$ satisfying \eqref{defWa} and \eqref{errWa2} for some $(\vea)_\alpha \in \mathcal{E}$ and let $\vpa \in K_\alpha^{\perp}$ satisfy
\be
 \Pi_{K_\alpha^{\perp}} \Bigg( \mWa + \vpa - (\triangle_g + 1)^{-1} \Big[ (1 - h_\alpha)(\mWa + \vpa) + f\big( \mWa + \vpa \big)  \Big] \Bigg) = 0
 \ee
 and 
 \[ \Vert \vpa \Vert_{H^1(M)} = o(1) \]
 as $\alpha \to + \infty$. Let $\ua = \mWa + \vpa$. Up to passing to a subsequence for $(\ua)_\alpha$ we have, for any sequence $(\xa)_\alpha$ of points of $M$,
 \ben \label{contreste}
\ua(\xa) = u_0(\xa) + \sum_{i=1}^k \Via(\xa) + o \Big( \sum_{i=0}^k \Bia(\xa) \Big),
  \een
  where $\Via$ is as in \eqref{Via} and $\Bia$ as in \eqref{defBia}.
 \end{theo}

Passing to a subsequence for $(\ua)_\alpha$ amounts to passing to a subsequence for $(\xia)_\alpha, (\mia)_\alpha$, $1 \le i \le k$, $(h_\alpha)_\alpha$, $(\vea)_\alpha$ and $(\vp_\alpha)_\alpha$. As will be momentarily proven, Theorem \ref{theorieC0} follows from Theorem \ref{theopsC0}. Theorem \ref{theopsC0} shows in particular that for approximate solutions $\mWa + \vpa$ to \eqref{eqha}, ie satisfying only \eqref{eqnonlin} and not necessarily \eqref{eqha}, the remainder $\vpa$ is globally pointwise small with respect to $u_0 + \sum_{i=1}^k \Via$. In particular, approximate solution obtained by a Lyapunov-schmidt method satisfy \eqref{contreste}. This applies to the examples of blowing-up solutions to second-order critical elliptic equations obtained in Ambrosetti-Malchiodi  \cite{AmbrosettiMalchiodi}, Brendle \cite{Brendle},  Brendle-Marques  \cite{BrendleMarques}, Esposito-Pistoia-V\'etois  \cite{EspositoPistoiaVetois}, Morabito-Pistoia-Vaira \cite{MorabitoPistoiaVaira}, Pistoia-Vaira \cite{PistoiaVaira},  Robert-V\'etois \cite{RobertVetois5, RobertVetois2,  RobertVetois}, Thizy \cite{ThizyLinNi} and Thizy-V\'etois \cite{ThizyVetois}. 
For strongly coupled systems or in nonsymmetric situations, pointwise estimates for approximate solutions constructed by a Lyapunov-schmidt method have proven more effective than $H^1(M)$ estimates to obtain the existence of blowing-up solutions, see e.g. Premoselli  \cite{Premoselli6, Premoselli7, Premoselli12} or Premoselli-Thizy \cite{PremoselliThizy}.

The approach that we follow here to prove Theorem \ref{theopsC0} (via the linear result of Theorem \ref{proplin} and the nonlinear perturbation argument of Proposition \ref{nonlin2}) is particularly robust and bypasses many difficulties that are usually encountered when trying to obtain pointwise bounds for solutions of critical nonlinear equations.

\begin{proof}
We first claim that exists $C >0$ and a sequence $(\sa)_\alpha$ satisfying \eqref{defsigmaa} such that, up to passing to a subsequence, 
\ben \label{C0ps1}
|\vpa(x)| \le C (\vea + \sa)^{\frac12} \sum_{i=0}^k \Bia(x) + C \sum_{i=1}^k \Tia(x) \Bia(x)
\een 
holds for all $x \in M$. Indeed, since $\mWa$ satisfies  \eqref{defWa} and \eqref{errWa2}, Propositions \ref{nonlin1} and \ref{nonlin2} both apply. On the one side, Proposition \ref{nonlin2} states, up to passing to a subsequence, the existence of $C>0$, of a sequence $(\sa)_\alpha$ satisfying \eqref{defsigmaa} and of a unique 
\be
\begin{aligned}
 \check{\vp}_\alpha \in K_\alpha^{\perp} \cap \Big \{ \vp \in C^0(M), \quad &|\vp(x)| \le C(\vea + \sa)^{\frac12} \sum_{i=0}^k \Bia(x) \\
 & + C \sum_{i=1}^k \Tia(x) \Bia(x) \quad \textrm{ for all } x \in M \Big \} 
 \end{aligned}
 \ee
that solves the equation
\ben  \label{thC02}
 \Pi_{K_\alpha^{\perp}} \Bigg( \mWa +  \check{\vp}_\alpha - (\triangle_g + 1)^{-1} \Big[ (1 - h_\alpha)(\mWa +  \check{\vp}_\alpha) + f\big( \mWa +  \check{\vp}_\alpha \big)  \Big] \Bigg) = 0
 \een
 and satisfies in addition $\Vert \check{\vp}_\alpha \Vert_{H^1(M)} \to 0$. On the other side, for some fixed $\delta >0$ small enough, Proposition \ref{nonlin1} states the existence of a unique $\hat{\vp}_\alpha \in K_\alpha^{\perp} \cap \{ \vp \in H^1(M), \Vert \vp \Vert_{H^1(M)} \le \delta \}$ satisfying \eqref{thC02} up to a subsequence. Since $\Vert \vpa \Vert_{H^1(M)} \to 0$ as $\alpha \to + \infty$ by assumption we therefore have $\vpa = \hat{\vp}_\alpha$ and $\check{\vp}_\alpha = \hat{\vp}_\alpha$. Hence $\vpa = \check{\vp}_\alpha $ and $\vpa$ satisfies \eqref{C0ps1}.

\medskip

We now prove \eqref{contreste}. We let $\ua = \mWa + \vpa$ and we will prove that, up to passing to a subsequence for $(\xia)_\alpha, (\mia)_\alpha$, $1 \le i \le k$, $(h_\alpha)_\alpha$, $(\vea)_\alpha$ and $(\vpa)_\alpha$, we have
\[ \Bigg | \Bigg| \frac{\ua - u_0 - \sum_{i=1}^k \Via}{\sum_{i=0}^k \Bia} \Bigg| \Bigg|_{L^\infty(M)} \to 0 \]
as $\alpha \to + \infty$. We proceed by contradiction and assume, up to passing to a subsequence, that there exists a sequence $(\xa)_\alpha$ of points of $M$ and $\eta_0 >0$ such that
\ben \label{C0ps11}
 \Big| \ua(\xa) - u_0(\xa) - \sum_{i=1}^k \Via(\xa)\Big|  \ge \eta_0 \sum_{i=0}^k \Bia(\xa). 
 \een
We still assume, up to renumbering the bubbles, that $\mu_{1, \alpha} \ge \mu_{2, \alpha} \ge \dots \ge \mu_{k,\alpha}$. By \eqref{defWa} and \eqref{C0ps1} such a sequence $(\xa)_\alpha$ then satisfies
\ben \label{locxa}
\frac{1}{\mu_{1, \alpha}} \min_{i=1\dots k} d_g(\xia, \xa) \to + \infty 
 \een
as $\alpha \to + \infty$.  By definition of $\vpa$  there exist real numbers $\lija$ such that $\vpa$ satisfies
\ben \label{eqvpabis}
 \big(\triangle_g + h_\alpha - f'(\mWa) \big)\vpa = \tilde{R}_\alpha + N_\alpha(\vpa) +  \sum \limits_{\underset{j =1 .. n_i }{i=1..k}} \lija ( \triangle_g   + 1)Z_{i,j, \alpha} 
 \een
 where $\tilde{R}_\alpha$ and $N_\alpha(\vpa)$ are as in \eqref{Rapreuve}. By \eqref{errWa2}, $\Vert \tilde{R}_\alpha\Vert_{H^1(M)} \le C( \vea + \sa)$ for some $C >0$. Since $\vpa = \check{\vp}_\alpha$ we have $\Vert \vpa \Vert_{H^1(M)} \le C(\vea + \sa)$ by \eqref{normereste}. Integrating \eqref{eqvpabis} against each $Z_{i,j,\alpha}$ and using \eqref{orthoZij} yields: 
\ben \label{C0ps2}
 |\lija| \le C( \vea + \sa)
 \een
for $0 \le i \le k$ and $1 \le j \le n_j$. Now, $\ua$ satisfies
\[ \triangle_g \ua + h_\alpha \ua = |\ua|^{2^*-2} \ua +  \sum \limits_{\underset{j =1 .. n_i }{i=1..k}} \lija ( \triangle_g   + 1)Z_{i,j, \alpha} . \]
Let $G_h$ be the Green's function of $\triangle_g + h$. Since $\Vert h_\alpha - h \Vert_{L^\infty(M)} \to 0$ as $\alpha \to + \infty$ a representation formula for $\triangle_g +h$ together with \eqref{C0ps2} yields, with \eqref{greenB}:
\ben \label{C0ps3}
\bal
 \ua(\xa) - \Pi_{K_h}(\ua)(\xa)= \int_M G_h(\xa, \cdot) |\ua|^{2^*-2} \ua dv_g \\
 + O \Big( (\vea + \sa) \sum_{i=0}^k \Bia(\xa) \Big)
 \eal
 \een
where we have let $K_h = \ker(\triangle_g + h)$. Assume first that $B_{0,\alpha} \equiv 0$, that is $u_0 \equiv 0$ and $\ker(\triangle_g +h) = \{0\}$. Then $\Pi_{K_h}(\ua) \equiv 0$ and, as a consequence of \eqref{C0ps1} and using \eqref{decbulle}, we have 
 \[  \int_{ M \backslash \bigcup_{i=1}^k   B_g(\xja, \frac{\ria}{R_0}) } G_h(x_\alpha, \cdot)^{2-n} |\ua|^{2^*-2} \ua dv_g  = O\Big( \sa \sum_{i=1}^k \Bia(x_\alpha) \Big).\]
 Assume now that $B_{0,\alpha}  \equiv 1$. Then, with \eqref{C0ps1} we have that 
 \[ \Pi_{K_h}(\ua)(\xa) = \Pi_{K_h}(u_0)(\xa) +O\Big( \sa \sum_{i=1}^k \Bia(x_\alpha) \Big). \]
On $M \backslash  \bigcup_{i=1}^k   B_g(\xja, \frac{\ria}{R_0}) $ we have by \eqref{defWa}, \eqref{proplin281}  and \eqref{C0ps1} that
 \[ \Big|  |\ua|^{2^*-2} \ua  -  |u_0|^{2^*-2} u_0 \Big| \lesssim \sum_{i=1}^k \Bia + \sum_{i=1}^k \Bia^{2^*-1} + \sa B_{0,\alpha} \]
for some sequence $(\sa)_\alpha$ satisfying \eqref{defsigmaa}. Since $B_{0,\alpha} \equiv 1$ we have $\ria \to 0$ for any $i \in \{1, \dots, k\}$ by definition, so using a representation formula for \eqref{limiteq}, \eqref{proplin281} and \eqref{greenB} below gives 
 \[ \bal
  \int_{ M \backslash \bigcup_{i=1}^k   B_g(\xja, \frac{\ria}{R_0}) } G_h(x_\alpha, \cdot)^{2-n} |\ua|^{2^*-2} \ua dv_g & = u_0(\xa) -  \Pi_{K_h}(u_0)(\xa) \\
  &+ O\Big( \sa \sum_{i=0}^k \Bia(x_\alpha) \Big).
  \eal \]
 Independently, if $i \in \{1, \dots, k\}$ we have, mimicking the proof of \eqref{proplin6}:
 \[
\bal  \int_{ \bigcup_{j \in \mathcal{B}_i}  \Big( B_g(\xja, \rjia) \backslash \cup_{s \in \mathcal{B}_i} B_g(x_{s, \alpha}, \frac{r_{s,\alpha}}{R_0}) \Big) } &G_h(x_\alpha, \cdot)^{2-n} |\ua|^{2^*-2} \ua dv_g  \\
& = O\Big( \sa \sum_{i=1}^k \Bia(x_\alpha) \Big).
\eal \]
By \eqref{defR0}, $B_g(\xia, \frac{\ria}{R_0})$ and $B_g(\xja, \frac{\rja}{R_0})$ are disjoint if $\mia \asymp \mja$. If $\mja = o(\mia)$ and $j \in \mathcal{B}_i$ then, since $\rjia \ge \sjia$ (see \eqref{defrhoia}) and $R_0 >2$, $\Oia$ and $B_g(\xja, \frac{\rja}{R_0})$ have empty intersection. Finally, if $\mja = o(\mia)$ and $ d_g(\xja, \xia)> \frac{2 \ria}{R_0}$ we have $\rja = o( d_g(\xia, \xja))$ by \eqref{defri}. Hence $B_g(\xia, \frac{\ria}{R_0})$ and $B_g(\xja, \frac{\rja}{R_0})$ are disjoint. This shows that the sets $\Oia$ defined in \eqref{defOia} are pairwise disjoint. Coming back to \eqref{C0ps3} we have thus shown that 
\ben \label{theorie0}
\bal
 \ua(x_\alpha) = u_0(\xa) +  \sum_{i=1}^k \int_{\Oia} G_h(x_\alpha, \cdot) |\ua|^{2^*-2} \ua dv_g \\
 + O \Big((\vea + \sa) \sum_{i=0}^k \Bia(x_\alpha) 
\Big).
\eal
 \een
Let $R >0$ and let $i \in \{1, \dots, k \}$. Let $y \in B(0,R) \cap \hat{\Omega}_{i,\alpha}$, where $\hat{\Omega}_{i,\alpha}$ is as in \eqref{defhOia}. By \eqref{locxa} we have
\[ (n-2) \omega_{n-1} d_g(\xia, x_\alpha)^{n-2} G_h \big(x_\alpha, \exp_{\xia}(\mia y) \big) \to 1 \quad \textrm{ if } d_g(\xia, x_\alpha) \to 0 \]
and
\[ G_h \big(x_\alpha, \exp_{\xia}(\mia y) \big) = G_h(x_\alpha, \xia) + o(1) \quad \textrm{ if } d_g(\xia, x_\alpha) \not \to 0\]
as $\alpha \to + \infty$. Since 
\[ (n-2) \omega_{n-1} d_g(\xia, x_\alpha)^{n-2} G_h(\xia, \xa) \to 1 \]
as $\alpha \to + \infty$ if $d_g(\xia, x_\alpha) \to 0 $ the dominated convergence theorem gives 
\ben \label{theorie3} \begin{aligned}
&\int_{\Oia \cap B_g(\xia, R\mia)}G_h(x_\alpha, \cdot) |\ua|^{2^*-2} \ua dv_g \\
& = (1+o(1)) G_h(x_\alpha, \xia) \mia^{\pui} \int_{B(0,R)} |V_i|^{2^*-2} V_i dx \\
& = (1+o(1)) G_h(x_\alpha, \xia)(n-2)\omega_{n-1} \mia^{\pui} \big( \lambda_i + O(R^{-2}) \big) \\
& =  \Via(x_\alpha) + o(\Bia(x_\alpha)) + O(R^{-2} \Bia(x_\alpha))
\end{aligned} \een
where the third line follows from \eqref{caracl} and the last one by \eqref{Via}. Independently, we have with \eqref{calcul4}, \eqref{C0ps1} and \eqref{decbulle} that 
\[ \Big| \int_{\Oia \backslash B_g(\xia, R\mia)}G_h(x_\alpha, \cdot) |\ua|^{2^*-2} \ua dv_g  \Big| \lesssim R^{-2} \Bia(\xa). \]
The last two lines together with \eqref{theorie0} show that, for any fixed $R >0$ and for any $\alpha$ large enough we have 
\[ \ua(x_\alpha) = u_0(\xa) +   \sum_{i=1}^k \Via(x_\alpha) + o\Big(\sum_{i=0}^k \Bia(x_\alpha)\Big) + O\Big(R^{-2} \sum_{i=1}^k  \Bia(x_\alpha) \Big). \]
Up to picking $R$ large enough and passing again to a subsequence this contradicts \eqref{C0ps11} and concludes the proof of Theorem \ref{theopsC0}.
\end{proof}

\subsection{Proof of Theorem \ref{theorieC0}}

In this last subsection we prove Theorem \ref{theorieC0}. We start with a simple consideration. Let $(M,g)$ be a closed manifold of dimension $n \ge 3$, let $(h_\alpha)_\alpha$ be a sequence of functions that converges in $L^\infty(M)$ towards $h$ as $\alpha \to + \infty$, $k \ge 1$ be an integer, $u_0 \in C^{1,\beta}(M)$ for $0 < \beta < 1$ be a solution of \eqref{limiteq}, $(\xia)_{\alpha}$ and $(\mia)_\alpha$, $1 \le i \le k$ be $k$ sequences satisfying \eqref{structure} and let $V_1, \dots, V_k $ in $D^{1,2}(\R^n)$ be $k$ solutions of \eqref{yamabe}.

\medskip

Let $(\vea)_\alpha \in \mathcal{E}$ and let  $c_{i,j,\alpha}$, $0 \le i \le k$, $1 \le i \le n_i$, be real numbers satisfying
\[ \sum_{i=0}^k\sum_{j=1}^{n_i} | c_{i,j,\alpha}| \le \vea. \]
We define, for any $\alpha \ge 1$, 
\ben \label{nouvdef}
\bal
u_{0, \alpha} & = 
u_0 + \sum_{j=1}^{n_0} c_{0,j,\alpha} Z_{0,j} \\\ 
W_{i,\alpha} & =  \Via + \sum_{j=1}^{n_i} c_{i,j,\alpha} Z_{i,j,\alpha}.
\eal
\een
Here $\Via$ is given by \eqref{Via}, $(Z_{i,j,\alpha})_{1 \le j \le n_i}$ for $1 \le i \le k$ are given by \eqref{defZij} and span $K_{i,\alpha}$ and $(Z_{0,j})_{1 \le j \le n_0}$ span $K_0$ (see \eqref{defK0} and \eqref{defKa}). If $u_0 \equiv 0$ and $\ker(\triangle_g +h) = \{0\}$ then $u_{0, \alpha} \equiv 0$. We let 
\[
\mathcal{W}_\alpha = u_{0,\alpha} + \sum_{i=1}^k W_{i,\alpha}.
\]
By \eqref{estV}, \eqref{nouvdef} and since $\vea \to 0$, this $\mWa$ satisfies \eqref{defWa}. 
As the following lemma shows it also satisfies \eqref{errWa2}: 

\begin{lemme} \label{lemmeerreur}
There exists a constant $C$ only depending on $(M,g),n,V_1, \dots, V_k$ and $u_0$ such that
\be
\bal
\Big| \triangle_g \mWa + h_\alpha \mWa &- |\mWa|^{2^*-2} \mWa \Big| \le C(\vea+\Vert h_\alpha - h \Vert_{L^\infty(M)}) B_{0,\alpha} \\
&+C  \sum_{i=1}^k \Bia + C  \sum \limits_{\underset{i \neq j}{i,j=0..k}}
 \Bia^{2^*-2} \Bja + C\vea \sum_{i=1}^k \Bia^{2^*-1}
\eal
\ee
in $M$. 
\end{lemme}
\begin{proof}
Let $\tilde{\mWa} = u_0 + \sum_{i=1}^k V_{i,\alpha}$. Using \eqref{defZij} we have 
\[ |\triangle_g Z_{j,m,\alpha}| \lesssim \Bja^{2^*-1} + \Bja \]
for $1 \le j \le k$ and $1 \le m \le n_j$, while if $j=0$ we have, for all $1 \le m \le n_0$, $|(\triangle_g +1)Z_{0,m,\alpha}| \lesssim 1$. With  \eqref{defWa} and \eqref{nouvdef} we then have, since $\vea \le 1$
\[ \big| (\triangle_g + h_\alpha)\mWa -  (\triangle_g + h_\alpha)\tilde{\mWa} \big| \lesssim \vea B_{0,\alpha} + \sum_{j=1}^k B_{j,\alpha} + \vea \sum_{j=1}^k B_{j,\alpha}^{2^*-1}.  \]
Independently, and since $\vea \le 1$,
\[ \big| |\mWa|^{2^*-2} \mWa -  |\tilde{\mWa}|^{2^*-2} \tilde{\mWa} \big| \lesssim \vea B_{0,\alpha} + \vea \sum_{j=1}^k B_{j,\alpha}^{2^*-1}.\]
Finally, \eqref{contVa} and \eqref{errVxa} show that 
\[ \begin{aligned}
 \Big| \triangle_g \tilde{\mWa} + h_\alpha \tilde{\mWa} - |\tilde{\mWa}|^{2^*-2} \mWa \Big| & \lesssim \Vert h_\alpha - h \Vert_{L^\infty(M)} B_{0,\alpha} \\
 +   \sum \limits_{\underset{i \neq j}{i,j=0..k}}\Bia^{2^*-2} \Bja + \sum_{j=1}^k B_{j,\alpha}. 
 \end{aligned}
 \]
 Gathering these estimates proves the lemma.
 \end{proof}

We can now conclude the proof of Theorem \ref{theorieC0}. 
 
 \begin{proof}[Proof of Theorem \ref{theorieC0}.]
Let $(M,g)$ be a closed Riemannian manifold of dimension $n \ge 3$ and $(h_\alpha)_\alpha$ be a sequence of functions that converges in $L^\infty(M)$ towards $h$ as $\alpha \to + \infty$. Let $(\ua)_\alpha$ be a sequence of (possibly sign-changing) solutions of 
\be
\triangle_g \ua + h_\alpha \ua = |\ua|^{2^*-2} \ua 
\ee
in $M$, satisfying \eqref{blowup}. Struwe's decomposition \eqref{struwe} shows that there exist $u_0 \in C^{1,\beta}(M)$ for some $0 < \beta < 1$ solving 
\[\triangle_g u_0 + h u_0 = |u_0|^{2^*-2} u_0, \]
and that there exist  $k \ge 1$, $k$ sequences $(\xia)_{\alpha}$, $1 \le i \le k$ and $(\mia)_\alpha$, $1 \le i \le k$ satisfying \eqref{structure} and $k$ solutions $V_1, \dots, V_k $ in $D^{1,2}(\R^n)$ of \eqref{yamabe} such that 
\be
 \ua = u_0 + \sum_{i=1}^k \Via + \Psi_\alpha, 
 \ee
where $\Vert \Psi_\alpha \Vert_{H^1(M)} \to 0$ as $\alpha \to + \infty$ and where $\Via = V_i^{\mia, \xia}$ is given by \eqref{Vxa}. Let
\[ \psi_\alpha^{\perp} = \Pi_{K_\alpha^{\perp}}(\Psi_\alpha)  \quad \textrm{ and } \quad \ \psi_\alpha^{||} = \Pi_{K_\alpha}(\Psi_\alpha), \] 
where $K_\alpha$ is given by \eqref{defKa}. By \eqref{orthoZij} the functions $Z_{i,j,\alpha}$, $0 \le i \le k$, $1 \le j \le n_i$ form an almost-orthonormal family of $K_\alpha$ for the $H^1(M)$ scalar product, and we therefore have
\ben \label{defcij}
 \psi_\alpha^{||} =  \sum \limits_{\underset{j =1 .. n_i }{i=0..k}} c_{i,j,\alpha} Z_{i,j,\alpha}
 \een
for some real numbers $c_{i,j,\alpha}$, $0 \le i \le k$, $1 \le i \le n_i$, satisfying
\[ \sum_{i=0}^k\sum_{j=1}^{n_i} | c_{i,j,\alpha}| \le C \Vert \Psi_\alpha \Vert_{H^1(M)}  \]
for some $C>0$ independent of $\alpha$. Let 
\ben \label{defea}
 \ve_\alpha = C \Vert \Psi_\alpha \Vert_{H^1(M)},
 \een
 so that $\vea \to 0$ as $\alpha \to + \infty$. We can then rewrite Struwe's decomposition as 
\be
 \ua = \mWa + \psa^{\perp}, 
 \ee
where $\psa^{\perp} \in K_\alpha^{\perp}$, $\mWa = u_{0,\alpha} + \sum_{i=1}^k W_{i,\alpha}$ and $u_{0,\alpha}$ and $\Wia$ are given by \eqref{nouvdef} with the value of $c_{i,j,\alpha}$ given by \eqref{defcij}. In particular, by Lemma \ref{lemmeerreur}, $\mWa$ satisfies \eqref{defWa} and \eqref{errWa2}. Since $\Vert \psa^{\perp}\Vert_{H^1(M)} \to 0$,  Theorem \ref{theopsC0} applies and concludes the proof. 
\end{proof}

\medskip

\begin{remark}
The proof of Theorem \ref{theopsC0} requires to prove the intermediate estimate \eqref{C0ps1}, which is not as precise as \eqref{contreste} at finite distances from the points $(\xia)_{1 \le i \le k}$. The additional terms $\Tia \Bia$ that have to be taken into account in Proposition \ref{nonlin2} come from the error $O(B_{\ma,\xa})$ in the right-hand side of \eqref{errVxa}. If $V$ is a sign-changing solution of \eqref{yamabe}, indeed, one only has in general
\[ |(\triangle_g  + h)V^{\ma, \xa} - |V^{\ma, \xa} |^{2^*-2} V^{\ma, \xa} | \lesssim B_{\ma, \xa} \quad \textrm{ in } M \]
where $V^{\ma,\xa}$ is given by \eqref{Vxa} and this error estimate does not seem to improve. This contrasts with what happens for positive solutions, where radiality ensures that
\[ \triangle_g B_{\ma,\xa} = B_{\ma, \xa}^{2^*-1} \]
at any order in conformal normal coordinates. 
\end{remark}

\begin{remark}
Theorem \ref{theorieC0} requires no assumption on $h$. If $\ker(\triangle_g + h) \neq \{0\}$ and $u_0 \equiv 0$ one expects $\ua$ to possess a non-zero projection on $\ker(\triangle_g +h)$. Under these assumptions we have $B_{0,\alpha} \equiv 1$ and this is the reason why $\ua - \sum_{i=1}^k \Via$ is estimated by $o \big( \sum_{i=0}^k \Bia \big)$ in this case.
\end{remark}

\begin{remark} \label{remprecision}
When all the bubbles $V_1, \dots, V_k$ appearing in the Struwe decomposition of $\ua$ are \emph{non-degenerate} solutions of \eqref{yamabe} it is likely that the estimates of Theorem \ref{theorieC0} can be improved. Sharper estimates in dimensions $n \ge 7$ for solutions of \eqref{eqha} looking like a tower of positive bubbles (with or without alternating signs) were for instance obtained in Premoselli \cite{Premoselli12} (sections $3$ and $4.2$).  If one of the $V_i$, $1 \le i \le k$, is degenerate then the possible component along the kernel in \eqref{defcij} seems difficult to control, and it is not clear whether one can expect better estimates than in Theorem \ref{theorieC0}. No examples of \emph{degenerate} solutions of \eqref{yamabe} are known, but their existence has not been ruled out yet. 
\end{remark}

\appendix

\section{Technical results} \label{technicalresults}

Throughout this appendix we use the notations of Section \ref{linear}. We let $(\mia)_{\alpha}$ and $(\xia)_\alpha$, $1 \le i \le k$, be sequences respectively of positive numbers and points of $M$ satisfying \eqref{structure} and we let $\Bia = B_{\mia, \xia}$ be given by \eqref{Ba}.

\subsection{Some estimates}

\begin{prop}
Let $n \ge 3$. The following estimates hold true: 

\begin{itemize}
\item Let $i, j \in\{1, \dots, k\}$, $i \neq j$. Then for any sequence $(\xa)_\alpha$ of points of $M$ we have 
\ben \label{greenB}
\int_M d_g(\xa,\cdot)^{2-n} \Bia dv_g \lesssim \Tia(\xa) \Bia(\xa)
\een
where $\Tia$ is as in \eqref{defTi}, and
\ben \label{greenBder}
\int_M d_g(\xa,\cdot)^{1-n} \Bia dv_g \lesssim \Xi_{i,\alpha}(\xa) \Bia(\xa) 
\een
where we have let, for $x \in M$,
\ben \label{defXia}
 \Xi_{i,\alpha}(x) =  \left \{ \begin{aligned} & \tia(x) \big| \ln \tia(x) \big| & \textrm{ if } n = 3 \\
&  \tia(x) & \textrm{ if } n \ge 4 \end{aligned} \right. ,
\een
and where $\tia$ is as in \eqref{defti}.

\item Let $i, j \in\{1, \dots, k\}$, $i \neq j$. There exists a sequence $(\sa)_\alpha$ as in \eqref{defsigmaa} such that for any sequence $(\xa)_\alpha$ of points in $M$ we have 
\ben \label{greenBij}
\int_M d_g(\xa,\cdot)^{2-n} \Bia^{2^*-2} \Bja dv_g \lesssim \sa \big( \Bia(\xa) + \Bja(\xa) \big)
\een
and
\ben \label{greenBijder}
\int_M d_g(\xa,\cdot)^{1-n} \Bia^{2^*-2} \Bja dv_g \lesssim \sa \big( \mia^\pui \tia(\xa)^{1-n}  + \mja^\pui \tja(\xa)^{1-n}  \big).
\een

\item  Let $i, j \in\{1, \dots, k\}$, $i \neq j$. There exists a sequence $(\sa)_\alpha$ as in \eqref{defsigmaa} such that 
\ben \label{gradij}
\int_M \mia^{\pui} \tia^{1-n} \mja^{\pui} \tja^{1-n} dv_g \lesssim \sa. 
\een

\item  Let $(\delta_\alpha)_\alpha$ be a sequence of positive numbers with $\delta_\alpha \ge \mia$. Then for any sequence $(\xa)_\alpha$ of points of $M$ we have 
\ben \label{decbulle}
\int_{M \backslash B_g(\xia, \delta_\alpha)} d_g(\xa, \cdot)^{2-n} \Bia^{2^*-1}dv_g \lesssim \left( \frac{\mia}{\delta_\alpha} \right)^2 \Bia(\xa). 
\een
In particular, if $\delta_\alpha \ge \ria$,
\[ \int_{M \backslash B_g(\xia, \delta_\alpha)} d_g(\xa, \cdot)^{2-n} \Bia^{2^*-1}dv_g \lesssim \sa \Bia(\xa) .\]
 \end{itemize}
In \eqref{greenB} -- \eqref{decbulle} the sequences $(\sa)_\alpha$ actually only depend on $(\mia)_\alpha$ and $(\xia)_\alpha$, $1 \le i \le k$. The notation ``$\;\lesssim\;$'' means ``$\;\le C \dots\;$'', where $C$ only depends on $n$ and $(M,g)$.

\end{prop}

\begin{proof}
These estimates follow from applications of Giraud's lemma and from \eqref{structure} and are obtained by straightforward computations. They can be found for instance in Hebey \cite{HebeyZLAM}, Chapter $7$. 
\end{proof}

\subsection{Proof of \eqref{prnonlin1}}

The following lemma was used in the proof of Proposition \ref{nonlin2}:

\begin{lemme} \label{lemtec}
Let $0 < \tau \le 2^*-1$ be fixed. There exists a sequence $(\sa)_\alpha$ as in \eqref{defsigmaa}, but depending only on $(\mia)_\alpha$ and $(\xia)_\alpha$, $1 \le i \le k$,  such that for all $x \in M$
\[  \Big( \sum_{i=1}^k \Bia(x) \Big)^{2^*-1-\tau}  \Big( \sum_{i=1}^k \Tia(x) \Bia(x) \Big)^{\tau} \le \sa \Big(   \sum_{i=0}^k \Bia(x) +  \sum_{i=1}^k \Bia(x)^{2^*-1} \Big)   \]
holds.
\end{lemme}

\begin{proof}
Let $x \in M$ and let  
\[ S_{\alpha,x} = \big \{ i \in \{1, \dots k\}, x \in B_g(\xia, \sqrt{\mia}) \big \}. \]
By \eqref{defTi} we have 
\[  \sum_{i \in S_{\alpha,x}} \Tia(x) \Bia(x) \le C\sqrt{\mu_{1, \alpha}} \sum_{i=1}^k \Bia(x)\]
for some $C = C(n) >0$. Again by \eqref{defTi}, we also have 
\[   \sum_{i \in \{1, \dots, k\} \backslash  S_{\alpha,x}} \Tia(x) \Bia(x) \le C\sqrt{\mu_{1, \alpha}}.\]
In the end, for all $x \in M$,
\[ \sum_{i =1}^k \Tia(x) \Bia(x) \le C\sqrt{\mu_{1, \alpha}} \sum_{i=1}^k \Bia(x) + C \sqrt{\mu_{1, \alpha}}, \]
so that 
\ben \label{lemtec1}
  \begin{aligned}
\Big( \sum_{i=0}^k \Bia(x) \Big)^{2^*-1-\tau} & \Big( \sum_{i=1}^k \Tia(x) \Bia(x) \Big)^{\tau} \\
& \le C\mu_{1, \alpha}^{\frac{\tau}{2}} \sum_{i=0}^k \Bia(x)^{2^*-1} +  C\mu_{1, \alpha}^{\frac{\tau}{2}}  \sum_{i=0}^k \Bia(x)^{2^*-1-\tau}  .
\end{aligned}
\een
Assume first that $0 < \tau \le 2^*-2$. Then $2^*-1-\tau \in [1, 2^*-1[$ and thus, for any $1 \le i \le k$,
\[ \begin{aligned}
\Bia(x)^{2^*-1-\tau} & \le C \left \{  \begin{aligned} & \Bia(x)^{2^*-1} & \textrm{ if } d_g(\xia, x) \le \sqrt{\mia} \\ & \Bia(x) &  \textrm{ if } d_g(\xia, x) > \sqrt{\mia}
\end{aligned} \right.  \\
 & \le C \big(\Bia(x)^{2^*-1} + \Bia(x) \big) 
 \end{aligned}
 \]
 for some $C = C(n) >0$. Since either $B_{0,\alpha} \equiv 0 $ or $B_{0,\alpha} \equiv 1 $, going back to \eqref{lemtec1} proves the lemma in this case with $\sa = C \mu_{1, \alpha}^{\frac{\tau}{2}}$. 
 
 Assume now that $2^*-2 < \tau \le 2^*-1$. By \eqref{defTi} there exists $C = C(n) >0$ such that $\Tia(x) \le C$ for all $1 \le i \le k$ and all $x \in M$. We can then write that 
 \[
  \begin{aligned}
& \Big( \sum_{i=0}^k \Bia(x) \Big)^{2^*-1-\tau}  \Big( \sum_{i=1}^k \Tia(x) \Bia(x) \Big)^{\tau} \\
&  = \Big( \sum_{i=0}^k \Bia(x) \Big)^{2^*-1-\tau}   \Big( \sum_{i=1}^k \Tia(x) \Bia(x) \Big)^{\tau - (2^*-2)} \Big( \sum_{i=1}^k \Tia(x) \Bia(x) \Big)^{2^*-2} \\
& \le C   \Big( \sum_{i=0}^k \Bia(x) \Big)\Big( \sum_{i=1}^k \Tia(x) \Bia(x) \Big)^{2^*-2},
\end{aligned}
\]
and the proof follows from the $\tau = 2^*-2$ case proven above.
\end{proof}

\bibliographystyle{amsplain}
\bibliography{biblio}

\end{document}